\input ./style/arxiv-vmsta.cfg
\documentclass[numbers,compress,v1.0.1]{vmsta}
\usepackage{vtexbibtags}

\volume{6}
\issue{3}
\pubyear{2019}
\firstpage{345}
\lastpage{375}
\aid{VMSTA139}
\doi{10.15559/19-VMSTA139}
\articletype{research-article}

\startlocaldefs
\newcommand{\lleft}{\left}
\newcommand{\rright}{\right}
\newtheorem{thm}{Theorem}
\newtheorem{lemma}{Lemma}
\newtheorem{cor}{Corollary}
\newtheorem{proposition}{Proposition}

\theoremstyle{definition}
\newtheorem{defin}{Definition}
\newtheorem{remark}{Remark}

\newcommand{\rrVert}{\Vert}
\newcommand{\llVert}{\Vert}
\newcommand{\rrvert}{\vert}
\newcommand{\llvert}{\vert}
\allowdisplaybreaks
\ifdefined\HCode 
\def\index#1{}
\else 
\fi
\endlocaldefs

\begin{document}

\begin{frontmatter}
\pretitle{Research Article}

\title{Spatial quadratic variations for the solution to a stochastic partial differential equation with elliptic divergence form operator}

\author{\inits{M.}\fnms{Mounir}~\snm{Zili}\thanksref{cor1}\ead[label=e1]{Mounir.Zili@fsm.rnu.tn}\orcid{0000-0002-6601-3926}}
\author{\inits{E.}\fnms{Eya}~\snm{Zougar}\ead[label=e2]{Eya.Zougar@fsm.rnu.tn}}
\thankstext[type=corresp,id=cor1]{Corresponding author.}
\address{\institution{University of Monastir},
Faculty of sciences of Monastir,
Department of Mathematics, LR18ES17,
Avenue de l'environnement,
5019 Monastir,
\cny{Tunisia}}



\markboth{M. Zili, E. Zougar}{Spatial quadratic variations for a SPDE's solution}

\begin{abstract}
We introduce a stochastic partial differential equation (SPDE) with
elliptic operator in divergence form, with measurable and bounded
coefficients and driven by space-time white noise. Such SPDEs could be
used in mathematical modelling of diffusion phenomena in medium
consisting of different kinds of materials and undergoing stochastic
perturbations. We characterize the solution and, using the
Stein--Malliavin calculus, we prove that the sequence of its recentered
and renormalized spatial quadratic variations satisfies an almost sure
central limit theorem. Particular focus is given to the interesting case
where the coefficients of the operator are piecewise constant.
\end{abstract}

\received{\sday{25} \smonth{4} \syear{2019}}
\revised{\sday{3} \smonth{9} \syear{2019}}
\accepted{\sday{3} \smonth{9} \syear{2019}}
\publishedonline{\sday{3} \smonth{10} \syear{2019}}
\end{frontmatter}

\section{Introduction}%
\label{sec1}
Many diffusion phenomena in 
various fields of real life are modelled by the
following type of partial differential equations\index{partial differential equations (PDE)} (PDEs\index{partial differential equations (PDE)})
%
\begin{equation}
\label{e:1} \frac{\partial u(t, x)}{\partial t} = {\mathcal{L}} u(t, x),
\end{equation}
where ${\mathcal{L}}$ is the elliptic divergence form operator defined
by
%
\begin{equation}
\label{e:1.1} {\mathcal{L}} = \frac{1}{r(x)} \frac{d}{dx}
\biggl( R(x) \frac{d}{dx} \biggr) ,
\end{equation}
$R$ and $r$ are two measurable and bounded functions defined on
${\mathbb{R}}$ and satisfying
\begin{equation*}
\mu _{1} \le R(x) \quad \mbox{and} \quad \mu _{2} \le
r(x) \quad \mbox{for all} \ x \in {\mathbb{R}}
\end{equation*}
where $\mu _{1}$ and $\mu _{2}$ are two strictly positive real constants,
and $\frac{d}{dx}$ denotes the derivative in the distributional sense. More
information on PDEs\index{partial differential equations (PDE)} of the type (\ref{e:1}) and their applications can
be found, e.g., in \cite{eco,Nic,Osaka} and references therein. One
interesting example of such PDEs\index{partial differential equations (PDE)} is the one defined by
%
\begin{equation}
\label{e:1.2} {\mathcal{L}}_{p} = \frac{1}{2\rho (x)}
\frac{d}{dx} \biggl( \rho (x)A(x) \frac{d}{dx} \biggr) ,
\end{equation}
%
\begin{equation}
\label{eq:coefA} A(x)= a_{1} {\mathbf{1}}_{\{ x \le 0 \} } +
a_{2} {\mathbf{1}}_{\{ 0 <
x \} } \quad \text{and} \quad \rho (x)=
\rho _{1} {\mathbf{1}}_{\{ x
\le 0 \} } + \rho _{2} {
\mathbf{1}}_{\{ 0 < x \} },
\end{equation}
$a_{i}, \rho _{i}$ ($i=1, 2$) are strictly positive constants. Operators
of the kind (\ref{e:1.2}) are the infinitesimal generators of diffusion
processes that have been widely studied in literature (see
\cite{ETOR,Lej} and references therein). The discontinuity of the
coefficients $A$ and $\rho $ reflects the heterogeneity of the media in
which the modelled process under study propagates.

In the present paper, we introduce a stochastic partial differential
equation\index{stochastic partial differential equation (SPDE)} (SPDE\index{stochastic partial differential equation (SPDE)}), that can be considered as a stochastic counterpart of
PDE (\ref{e:1}). More specifically, we consider
%
\begin{equation}
\label{e:2g} \lleft \{ %
\begin{array}{@{}rcl}
\displaystyle
\frac{\partial u(t, x)}{\partial t} &=&
\displaystyle
{\mathcal{L}} u(t, x) +\dot{W} (t,x) ;
\quad
t > 0,  x\in { \mathbb{R}},
\\
\noalign{\vskip2mm}
\displaystyle
u(0,.) &:= &0,\;
\end{array} %
\rright .
\end{equation}
where ${\mathcal{L}}$ is defined by (\ref{e:1.1}) and $\dot{W}$ denotes
the formal derivative of a space-time white noise. That is, $W$ is a
centered Gaussian field $ W = \{ W(t,C); t \in [0,T ], C \in B_{b}(
{\mathbb{R}}) \}$ with covariance
%
\begin{equation}
\label{e:3_} {\mathbb{E}} \bigl(W(t,C)W(s,D) \bigr) = (t \wedge s) \lambda (C
\cap D),
\end{equation}
where $ \lambda $ denotes the Lebesgue measure and $B_{b}({\mathbb{R}})$
is the set of $\lambda $-bounded Borel sub-sets of ${\mathbb{R}}$. So
$W$ behaves as a Wiener process\index{Wiener  process} both in time and in space. The solution
to Equation (\ref{e:2g}) is a random field $\{ u(t,x), t \ge 0, x
\in {\mathbb{R}} \}$, where $t$ represents the time variable and $x$ is
the space variable. In the particular case where the functions $r$ and
$R$ are constants $r:= 2$ and $R:= 1$, the operator ${\mathcal{L}} $ is
reduced to $\frac{1}{2} \frac{\partial ^{2}}{\partial x^{2}}$. So
Equation (\ref{e:2g}) also represents a natural extension of the
stochastic heat equation\index{stochastic heat equation} driven by the space-time white noise, which has
been widely studied in the literature (see \cite{Tud1} and the
references therein). This can be considered as an important motivation
for the investigation of such equation's solution.

This paper has a twofold objective: the first is to lay the first
milestone towards the investigation of the stochastic process solution
to (\ref{e:2g}). We prove its existence, and\vadjust{\goodbreak} we investigate its spatial
quadratic variation. In fact, the study of quadratic variation is
motivated by its numerous applications in many fields. For example, in
the estimation theory, the analysis of the asymptotic behaviour of the
quadratic variations of self-similar processes play an important role
in the construction of consistent estimators for the self-similarity
parameter (see, \xch{e.g.,}{e.g} \cite{Tudor-Viens} and references therein). In
stochastic analysis, quadratic variations are as well one of the main
tools used to characterize the semi-martingale property for some mixed
Gaussian processes\index{Gaussian process} (see, e.g., \cite{EZ,MishuraZili,MZ2}). Examples
of applications of quadratic variation investigation include also the
theory of the It\^{o} stochastic calculus with respect to martingales
\cite{Revuz} and mathematical finance \cite{Bardorff}. We refer
to the monograph \cite{Tud1} for a more complete exposition on
variations of stochastic processes in general, and of solutions to
certain SPDEs in particular. In this paper, under some conditions on the
fundamental solution\index{fundamental solution} to PDE (\ref{e:1}), we fix $t$ and 
study the
limit behaviour in distribution of the sequence $( \sum_{j=0} ^{N-1} ( u(t, \frac{j+1}{N}) - u(t,
\frac{j}{N})) ^{2} )_{N \ge 1}$. More precisely, using some
elements of the Stein--Malliavin calculus, we show that, after
recentralization and renormalization, the above sequence satisfies an
Almost Sure Central Limit\index{almost sure central limit theorem (ASCLT)} Theorem (in short: ASCLT\index{almost sure central limit theorem (ASCLT)}). Similar study has
been done in the case of stochastic heat equation\index{stochastic heat equation} (see
\cite{R,TT}) and also in the case of stochastic wave equation (see
\cite{CAZilKa}). But no similar study has been carried out on SPDEs\index{stochastic partial differential equation (SPDE)}
(\ref{e:2g}). For more information on ASCLT,\index{almost sure central limit theorem (ASCLT)} see
\cite{Azmoodeh-Nourdin} and references therein. The second objective of
this paper is to make a further study of the SPDE\index{stochastic partial differential equation (SPDE)} defined by\looseness=1
%
\begin{equation}
\label{e:2p} \lleft \{ %
\begin{array}{@{}rcl}
\displaystyle
\frac{\partial u(t, x)}{\partial t} &=&
\displaystyle
{\mathcal{L}}_{p} u(t, x) +\dot{W} (t,x) ;
\quad
t > 0,  x\in { \mathbb{R}},
\\
\noalign{\vskip2mm}
\displaystyle
u(0,.) &:= &0,\;
\end{array} %
\rright .
\end{equation}\looseness=0
where ${\mathcal{L}}_{p}$ is the operator defined by (\ref{e:1.2}). We
note that Equation (\ref{e:2p}) is a particular case of (\ref{e:2g}),
and it could be a good model for diffusion phenomena in a medium
consisting of two kind of materials, undergoing stochastic
perturbations. Equation (\ref{e:2p}) has been introduced in
\cite{ZZ} where the authors proved the existence of the solution and
they presented explicit expressions of its covariance and variance
functions. Some regularity properties of the solution sample paths have
also been analyzed. In \cite{ZZ2}, Zili and Zougar expanded the
quartic variations in time and the quadratic variations in space of the
solution to Equation (\ref{e:2p}). Both expansions allowed them to
deduce an estimation method of the parameters $a_{1}$ and $a_{2}$
appearing in (\ref{eq:coefA}). We make here another step in the study
of SPDE\index{stochastic partial differential equation (SPDE)} (\ref{e:2p}) by showing that its solution satisfies all
conditions under which we can use the ASCLT.\index{almost sure central limit theorem (ASCLT)} In addition to the
Stein--Malliavin calculus, our proofs require many integration
techniques, calculation, and analysis tools.\looseness=1

The paper is organized as follows. In the next section, we prove the
existence of the mild solution\index{mild solution} to Equation (\ref{e:2g}) and we give some
characterizations of its spatial increments. In \xch{Section \ref{sec3}}{section $3$}, using some
elements of the Stein--Malliavin theory, we establish an almost sure central
limit\index{almost sure central limit theorem (ASCLT)} theorem 
which applies to the solution to SPDE\index{stochastic partial differential equation (SPDE)} (\ref{e:2g}) under some
conditions on the fundamental solution\index{fundamental solution} associated to the operator
${\mathcal{L}}$. The last section focuses on a further investigation of
the solution to SPDE\index{stochastic partial differential equation (SPDE)} (\ref{e:2p}).
In which case,
we show that
the ASCLT\index{almost sure central limit theorem (ASCLT)}
is statisfied.

\section{Existence and some characteristics of the solution}%
\label{sec2}
The notion of solution to (\ref{e:2g}) is defined in the mild sense.
We call a mild solution\index{mild solution} to (\ref{e:2g}) the stochastic process
%
\begin{equation}
\label{sol} u(t,x)= \int_{0} ^{t} \int
_{\mathbb{R} } G(t-s, x, y) W(ds, dy), \quad t\in [0,T], x\in
\mathbb{R},
\end{equation}
where $W$ is the Gaussian noise with covariance given by (\ref{e:3}),
$G$ is the fundamental solution\index{fundamental solution} of the operator ${\mathcal{L}}$ and the
integral in (\ref{sol}) is the Wiener integral\index{Wiener ! integral} with respect to the
Gaussian noise $W$. The existence and many properties of the fundamental
solution $G$\index{fundamental solution} of the operator ${\mathcal{L}}$ have been obtained in many
papers (see, e.g., \cite{Lejay} and \cite{Osaka}).

\begin{remark}
\label{remark1}
It is well known (see, e.g., \cite{walsh}) that the mild solution\index{mild solution}
to (\ref{e:2g}) exists when the Wiener integral\index{Wiener  integral} (\ref{sol}) is
well-defined and this happens when the function $(s,y) \longmapsto G(t-s, x, y)$
belongs to ${\mathcal{H}}_{0}=
L^{2}([0,T ] \times {\mathbb{R}})$, the canonical Hilbert space
associated with the Gaussian process\index{Gaussian process} $W$. In fact, ${\mathcal{H}}_{0}$
is none other than the closure of the linear span\index{linear span} generated by the
indicator functions $\mathbf{1}_{[0,t] \times C }$,\index{indicator functions} $ t \in [0,T ]$, $C
\in {\mathcal{B}}_{b}({\mathbb{R}})$, with respect to the inner product
\begin{equation*}
{<} 1_{[0,t]\times C}, 1_{[0,s] \times D}{>}_{{\mathcal{H}}_{0}} = (t \wedge s)
\lambda ( C \cap D).
\end{equation*}
Moreover, the process $(u(t, x), t \in [0,T ], x \in {\mathbb{R}})$, when it exists, is a
centered Gaussian process.
\end{remark}

The following proposition deals with the existence of the mild solution\index{mild solution}
to Equation (\ref{e:2g}).
%
\begin{proposition}
\label{exSol}
The centered Gaussian process $(u(t, x), t \in [0,T ], x \in {\mathbb{R}})$ defined by \textup{(\ref{sol})},
as a solution to Equation \textup{(\ref{e:2g})}, exists and satisfies
\begin{equation*}
\sup_{t \in [0,T], x \in {\mathbb{R}}} {\mathbb{E}} \bigl( u(t,x)^{2} \bigr)
< + \infty .
\end{equation*}
\end{proposition}
\begin{proof}
The existence and some bounds of the fundamental solution\index{fundamental solution} to PDE
(\ref{e:1}) have been established in \cite{Aronson} and
\cite{Dan}. In particular, it has been proved that there exist positive
constants $C_{1}$ and $C_{2}$ such that
%
\begin{equation}
G(t,x, y) \le \frac{C_{1}}{\sqrt{2 \pi t}} \exp \biggl( -\frac{C_{2}
(x-y)^{2}}{t} \biggr),
\end{equation}
for any $t \in [0,T]$ and $(x,y)\in {\mathbb{R}}^{2}$. Thus,
\begin{eqnarray*}
\int_{0}^{t} \int
_{\mathbb{R}} G^{2}(t-s,x, y) dy ds &\le & \int
_{0}^{t} \int_{\mathbb{R}}
\frac{C_{1}^{2}}{2 \pi (t-s)} \exp \biggl( -\frac{2C_{2} (x-y)^{2}}{t-s} \biggr) dy ds
\\
&\le & C_{3} \int_{0}^{t}
\frac{1}{\sqrt{t-s}} ds
\\
& \le & C_{4} \sqrt{T},
\end{eqnarray*}
where $C_{3}$ and $C_{4}$ denote two strictly positive constants. This
with Remark \ref{remark1} and Wiener\index{Wiener}'s isometry\index{isometry} allow us to get the
existence of the mild solution\index{mild solution} to (\ref{e:2g}) and to show that
\begin{equation*}
{\mathbb{E}} \bigl( u(t,x)^{2} \bigr) = \int_{0}^{t}
\int_{\mathbb{R}} G ^{2}(t-s,x, y) dy ds \le
C_{4} \sqrt{T},
\end{equation*}
for every $t \in [0,T]$ and $x\in {\mathbb{R}}$.
\end{proof}

Now we consider an interval $I$  in ${\mathbb{R}}$ and denote
\begin{equation*}
\Delta _{h}G(u,x,z)= G(u,x+h,z)-G(u,x,z)
\end{equation*}
and
\begin{equation*}
\bigl\llVert \Delta _{h}G(t-.,x,.) \bigr\rrVert
^{2}_{L^{2}([0,t]\times {
\mathbb{R}})}= \int_{0}^{t}
\int_{{ \mathbb{R}}} \bigl( \Delta _{h}G(t- \sigma ,x,y)
\bigr)^{2}\,d\sigma \,dy
\end{equation*}
for every $u, t \in (0,T]$, $h > 0$ and $x, z \in I $. We also consider
the conditions:

\setlength\leftmargini{33pt}
\begin{enumerate}
\item[${H_{1}(I)}$:] $\forall t \in (0, T], \; \exists \; C_{5} > 0; \; \forall (x, y)
\in I^{2}; \; y > x,$
\begin{equation*}
C_{5} (y-x) \le \bigl\llVert \Delta _{y-x}G(t-.,x,.)
\bigr\rrVert _{L^{2}((0,t)\times {\mathbb{R}})} ^{2}.
\end{equation*}

\item[${H_{2}(I)}$:]
$\forall t \in (0, T], \; \exists \; C_{6} > 0; \; \forall (x, y)
\in I^{2}; \; y > x,$
\begin{equation*}
\bigl\llVert \Delta _{y-x}G(t-.,x,.) \bigr\rrVert
_{L^{2}((0,t)\times {\mathbb{R}})}^{2} \le C_{6} (y-x).
\end{equation*}

\item[${H_{3}(I)}$:]
$\forall t \in (0, T], \exists \; C_{7} > 0; \; \forall h > 0, \;
\forall (x, y) \in I^{2}$,
\begin{equation*}
\int_{0}^{t} \int_{\mathbb{R}}
\Delta _{h}G(t-s,x,z)\Delta _{h}G(t-s,y,z)\,ds\,dz
\leqslant C_{7} \; \, h^{2}.
\end{equation*}
\end{enumerate}

The following lemma will play an important role in this paper.
%
\begin{lemma}
\label{th:spatial-reg}
Let u be the mild solution\index{mild solution} to Equation \textup{(\ref{e:2g})}.
\begin{enumerate}%
\item
If Condition $H_{1}(I)$ is satisfied then, for every $t > 0$, there
exists a positive constant $C_{8}$ such that
%
\begin{equation}
\label{eq:Holder} \forall x, y \in I, \quad C_{8} \, \llvert y-x
\rrvert \le \mathbb{E} \bigl( u(t,y)-u(t,x) \bigr) ^{2} .
\end{equation}%
\item
If Condition $H_{2}(I)$ is satisfied then, for every $t > 0$, there
exists a positive constant $C_{9}$ such that
%
\begin{equation}
\label{eq:Holder2} \forall x, y \in I, \quad \mathbb{E} \bigl( u(t,y)-u(t,x)
\bigr) ^{2} \leqslant C_{9} \, \llvert y-x \rrvert .
\end{equation}%
\item
If Condition $H_{3}(I)$ is satisfied then, for every $t > 0$,
$\forall h > 0$,
%
\begin{equation}
\label{Maj} \forall \; x, y \in I, \quad \mathbb{E} \bigl( \bigl(
u(t,x+h)-u(t,x) \bigr) \bigl( u(t,y+h)-u(t,y) \bigr) \bigr) \leqslant \,
C_{7} h ^{2}.
\end{equation}
\end{enumerate}
\end{lemma}
\begin{proof}
We first note that if $x = y$, then Inequalities (\ref{eq:Holder}) and
(\ref{eq:Holder2}) are trivial. We also note that the proofs in the\vadjust{\goodbreak}
cases $x > y $ and $y > x$ are similar. So we consider only the case
$y > x$. Using Wiener\index{Wiener}'s isometry\index{isometry} we get
%
\begin{eqnarray}
\label{eq:Iso-Sol} %
&& \mathbb{E} \bigl( u(t,y)-u(t,x)
\bigr)^{2}
\nonumber
\\[-2pt]
&=& \mathbb{E} \biggl(\int_{(0,t)\times { \mathbb{R}}}G(t-u,y,z)W(du,dz)\,- \int
_{(0,t)\times { \mathbb{R}}}G(t-u,x,z) W(du,dz) \biggr)^{2}
\nonumber
\\[-2pt]
&=& \mathbb{E} \biggl(\int_{(0,t)\times { \mathbb{R}}} \bigl(
G(t-u,y,z)-G(t-u,x,z) \bigr) W(du,dz) \biggr)^{2}
\nonumber
\\[-2pt]
&=& \int_{(0,t)\times { \mathbb{R}}} \bigl( G(t-u,y,z)-G(t-u,x,z)
\bigr)^{2} du \,dz
\nonumber
\\[-2pt]
&=& \bigl\llVert \Delta _{ y-x }G(t-.,x,.) \bigr\rrVert
^{2}_{L^{2}([0,t]\times
{ \mathbb{R}})} . %
\end{eqnarray}
Equality (\ref{eq:Iso-Sol}) and Condition $H_{1}(I)$ [respectively
$H_{2}(I)$] allow us to get the two first assertions in Lemma
\ref{th:spatial-reg}.

As for the third one, using again  Wiener\index{Wiener}'s isometry\index{isometry} we get
\begin{eqnarray*}
&&\mathbb{E} \bigl( \bigl( u(t,x + h)-u(t,x) \bigr) \bigl(
u(t,y+h)-u(t,y) \bigr) \bigr)
\\[-2pt]
&=&\mathbb{E} \biggl( \int_{(0,t)\times { \mathbb{R}}}\Delta _{h}G(t-u,x,z)W(du,dz)
\\[-2pt]
&& {}\times \int_{(0,t)\times { \mathbb{R}}}\Delta _{h}G(t-u,y,z)
W(du,dz) \biggr)
\\[-2pt]
&=& \int_{(0,t)\times {
\mathbb{R}}}\Delta _{h}G(t-u,x,z)\,\Delta
_{h}G(t-u,y,z) \;du\,dz.
\end{eqnarray*}
Using Condition $H_{3}(I)$ the proof of the third assertion in Lemma
\ref{th:spatial-reg} is achieved.
\end{proof}

From Assertion $2$ in Lemma \ref{th:spatial-reg} and by Kolmogorov's
criterion of continuity, we easily get the following corollary.

\begin{cor}
\label{Hold-cont}
Let $u$ be the mild solution\index{mild solution} to \textup{(\ref{e:2g})}. If Condition $H_{2}(I)$
is satisfied, then, for every $t \in [0, T]$, the process $( u(t,x) )_{x \in I}$ is H\"{o}lder continuous of order
$\gamma $ with $0 < \gamma <\frac{1}{2}$.
\end{cor}

\begin{remark}
From Corollary \ref{Hold-cont}, under Condition $H_{2}(I)$, the process
$u$ being a solution to (\ref{e:2g}) keeps the same H\"{o}lder regularity in
space as the solution to the standard stochastic heat equation\index{stochastic heat equation} driven
by a time-space white noise (see \cite{R} and references therein).
\end{remark}

\section{Almost sure central limit\index{almost sure central limit theorem (ASCLT)} theorem}%
\label{sec3}
Let us start this section with the following definition.

\begin{defin}
Let $(G_{N})_{N \ge 1}$ be a sequence of real-valued random variables
defined on a common probability space $(\varOmega , {\mathcal{F}},
{\mathbb{P}})$.\index{probability space} We say that the sequence $(G_{N})_{N \ge 1}$ satisfies
an almost sure central limit\index{almost sure central limit theorem (ASCLT)} theorem\index{almost sure central limit theorem (ASCLT)} (ASCLT\index{almost sure central limit theorem (ASCLT)}), if, almost surely, for
every bounded and continuous function $\varphi :\mathbb{R}\rightarrow
\mathbb{R}$, we have:
\begin{equation*}
\frac{1}{logN}\sum_{i=1}^{N}
\frac{\varphi (G_{i})}{i} {\longrightarrow } \mathbb{E} \bigl(\varphi ( \mathcal{Z})
\bigr)\quad \mbox{as}\ N\longrightarrow \infty ,
\end{equation*}
where ${\mathcal{Z}}$ is an ${\mathcal{N}}(0,1)$ random variable.
\end{defin}

For fixed $t \in (0, T]$, we consider the Gaussian process\index{Gaussian process} $( u(t,x) )_{x \in [0,1]}$ being the mild solution\index{mild solution} to Equation
(\ref{e:2g}). We also consider the partition $
0= x_{0} < x_{1} < \cdots < x_{N} = 1$ of the interval
$[0,1]$ defined by $
x_{i} = \frac{i}{N}$ for every $i = 0,1,\ldots, N $. We define the centered
re-normalized quadratic variation statistic in the following way:
%
\begin{equation}
\label{var} V_{N}= \displaystyle \sum
_{i=0}^{N-1} \biggl[ \displaystyle
\frac{   ( u(t,x_{i+1}) -
u(t,x_{i})  )^{2} }{ \mathbb{E}   ( u(t,x
_{i+1}) -
u(t,x_{i})  )^{2} } -1 \biggr] \quad \mbox{and} \quad \tilde{V}_{N}=
\frac{1}{\sqrt{2N}}\,V_{N}.
\end{equation}

The aim of this section is to show that the sequence $(\tilde{V}_{N})_{N
\ge 1}$ satisfies the ASCLT.\index{almost sure central limit theorem (ASCLT)} Let us first recall briefly some basic
elements of the Stein--Malliavin theory (see \cite{nou}) that will
be useful in our proof.

\subsection{Elements of the \xch{Stein--Malliavin}{Stein- Malliavin} theory}%
\label{sec3.1}
Consider a real separable Hilbert
space $(\mathcal{H}, {<}.,.{>}_{\mathcal{H}})$  and an isonormal
Gaussian process\index{isonormal Gaussian process} $ (B(\varphi ),\varphi \in \mathcal{H} )$ on a probability space $(\varOmega ,\mathcal{F},
\mathbb{P})$,\index{probability space} which is a centered Gaussian family of random variables
such that
\begin{equation*}
\mathbb{E} \bigl(B(\varphi ),B(\psi ) \bigr)= {<}\varphi ,\psi
{>}_{
\mathcal{H}},
\end{equation*}
for every $\varphi ,\psi \in \mathcal{H}$. For $q \ge 1$, let
${\mathcal{H}}^{\otimes q}$ be the $q$th tensor product of
${\mathcal{H}}$ and denote $\mathcal{H}^{\odot q}$ the associated
$q$th symmetric tensor product.

Denote by $I_{q}$ the $q$th multiple stochastic integral with respect to
$B$. This $I_{q}$ is actually an isometry\index{isometry} between the Hilbert space
$\mathcal{H}^{\odot q}$ equipped with the scaled norm $\frac{1}{\sqrt{q!}}\; \llVert   . \rrVert   _{\mathcal{H}^{\otimes q}}
$ and the Wiener chaos\index{Wiener  chaos} of order $q$, which is defined as the closed
linear span\index{linear span} of the random variables $H_{q}(B(\varphi ))$, where
$\varphi \in \mathcal{H},\;\;  \llVert   \varphi  \rrVert   _{
\mathcal{H}}=1$ and $H_{q}$ is the Hermite polynomial of degree
$q\geqslant 1$ defined by
%
\begin{equation}
H_{q}(x)= \frac{(-1)^{q}}{q!}\,\exp \biggl( \frac{x^{2}}{2}
\biggr)\, \frac{d^{q}}{dx^{q}} \biggl(\exp \biggl( -\frac{x^{2}}{2} \biggr)
\biggr), \quad x\in { \mathbb{R}}.
\end{equation}
The isometry\index{isometry} of multiple integrals can be written as follows: for $p,\,q
\geqslant 1,f\in \mathcal{H}^{\otimes p} $ and $ g\in \mathcal{H}^{
\otimes q} $,
%
\begin{equation}
\label{Isom} \mathbb{E} \bigl( I_{p}(f)I_{q}(g)
\bigr)= %
\begin{cases}
q!{<}\hat{f},\hat{g}{>}_{ \mathcal{H}^{\otimes q} }
&
\mbox{if}\,p=q
\\
0
&
\mbox{otherwise}.
\end{cases} %
\end{equation}
It holds that
\begin{equation*}
I_{q}(f)=I_{q}(\hat{f}),
\end{equation*}
where $\hat{f}$ denotes the canonical symmetrization of $f$ defined by
%
\begin{equation}
\hat{f}(x_{1},\ldots,x_{q})=\frac{1}{q!}\sum
_{\sigma \in S_{q}}f(x_{
\sigma (1)},\ldots,x_{\sigma (q)}),
\end{equation}
where the sum runs over all permutations $\sigma $ of $\{ 1,\ldots,q\}$.

We recall that any square-integrable random variable $F$, which is
measurable with respect to the $\sigma $-algebra generated by $B$, can be\vadjust{\goodbreak}
expanded into an orthogonal sum of multiple stochastic integrals:\index{stochastic integrals}
%
\begin{equation}
F= \mathbb{E}(F)+ \sum_{q=1}^{\infty }I_{q}(f_{q}),
\end{equation}
where the series converges in the $L^{2}(\varOmega )$-sense and the kernels
$f_{q}$, belonging to $\mathcal{H}^{\odot q}$, are uniquely determined
by $F$.

Consider now the class of smooth random variables $F$ that can be
written in the form
%
\begin{equation}
\label{smooth} F = g \bigl(B(\varphi _{1}), \ldots , B(\varphi
_{n}) \bigr),
\end{equation}
where $n\geqslant 1$, $g:{ \mathbb{R}}^{n}\longmapsto { \mathbb{R}}$ is
a ${\mathcal{C}}^{\infty }$-function with compact support and
$\varphi _{1}, \ldots, \varphi _{n} \in \mathcal{H}$. The Malliavin
derivative of a smooth random variable $F$ of the form (\ref{smooth})
is the $\mathcal{H}$-valued random variable given by
%
\begin{equation}
\label{DM} DF= \sum_{i=1}^{n}
\frac{\partial g}{\partial x_{i}} \bigl(B(\varphi _{1}), \ldots ,B(\varphi
_{n}) \bigr)\,\varphi _{i}.
\end{equation}

The following formula for multiplication of Wiener chaos\index{Wiener  chaos} integrals of
any orders $p,q$ will play a basic role in the next section. For any
symmetric integrands $f\in \mathcal{H}^{\odot p}$ and $g\in
\mathcal{H}^{\odot p}$, we have
%
\begin{equation}
\label{prodMI} I_{p}(f)I_{q}(g)= \sum
_{r=0}^{p\wedge q}r! \bigl(^{p}_{r}
\bigr) \bigl(^{q} _{r} \bigr)I_{p+q-2r}(f\otimes
_{r} g),
\end{equation}
where, in the particular case when $\mathcal{H}=L^{2}  ( [0,T]
  )$, for $r= 1,\ldots, p\wedge q$, the $r$th contraction $f \otimes
_{r} g$ is the element of $\mathcal{H}^{\otimes (p+q-2r)}$ defined by
%
\begin{equation}
\label{contraction} %
\begin{array}{@{}rcl}
&&
\displaystyle
(f\otimes _{r} g)(s_{1},\ldots,s_{p-r},t_{1},\ldots,t_{q-r})
\\
\noalign{\vskip2mm}
&=&
\displaystyle
\int_{[0,T]^{r}}du_{1}...du_{r} f(s_{1},\ldots,s_{p-r},u_{1},\ldots,u_{r})
\,g(t_{1},\ldots,t_{q-r},u_{1},\ldots,u_{r}).
\end{array} %
\end{equation}

The following theorem gives a description of the normal approximation
of multiple stochastic integrals.\index{stochastic integrals} We refer to
\cite{nou,NGP2012,NOL} and references therein for the proof.
%
\begin{thm}
\label{theo-inter}
Fix $q\geqslant 1$. Assume that $(G_{N})_{N\geqslant 1}:=   (\mathrm{I}
_{q}(g_{N})  )_{N\geqslant 1}$ with $g_{N} \in \mathcal{H}^{
\odot q} $ is a sequence of random variables belonging to the $q$th Wiener
chaos\index{Wiener chaos} such that
\begin{equation*}
\lim_{N \rightarrow \infty }\mathbb{E} \bigl(G^{2}_{N}
\bigr) = \sigma ^{2}.
\end{equation*}
Hence, $G_{N}$ converges in law to $\mathcal{Z}\sim \mathcal{N}(0,1)$
if and only if
\begin{equation*}
\lim_{N \rightarrow \infty } \llVert DG_{N} \rrVert
^{2}_{
\mathcal{H}} = q\sigma ^{2}.
\end{equation*}
Furthermore, if we denote by $d$ one of the metrics on the space of
probability measures on ${\mathbb{R}}$, including the Kolmogorov,\index{Kolmogorov}
Wasserstein\index{Wasserstein} and Total Variation measures, then for $N$ large enough:
\begin{equation*}
d \bigl( G_{N},\mathcal{N}(0,1) \bigr)\leqslant C \bigl( \sqrt{
\mathbf{Var} \bigl( \llVert DG_{N} \rrVert ^{2}_{\mathcal{H}}
\bigr)} \,+\, \sqrt{ \mathbb{E} \bigl( \llVert DG_{N} \rrVert
^{2}_{
\mathcal{H}} \bigr)-q\sigma ^{2} } \bigr) .
\end{equation*}
\end{thm}

The following theorem has been introduced in \cite{Nua}. It gives
a sufficient condition for extending Theorem \ref{theo-inter} to an
ASCLT\index{almost sure central limit theorem (ASCLT)} for multiple stochastic integrals.\index{stochastic integrals}

\begin{thm}
\label{T:7}
Fix $q\geqslant 2$, and let $(G_{N})_{N\geqslant 1}$ be a sequence of random
variables defined by
\begin{equation*}
G_{N}:= \bigl(\mathrm{I}_{q}(g_{N})
\bigr)_{N\geqslant 1}; \quad g_{N} \in \mathcal{H}^{\odot q}.
\end{equation*}
Suppose that:
\begin{enumerate}%
\item
For every $
\displaystyle
N\geqslant 1,\; \mathbb{E}(G^{2}_{N}) = 1$.
\item
For every $ r= 1,\ldots,q-1$, $
\displaystyle
\lim_{N \rightarrow \infty } \| g_{N} \otimes _{r} g_{N}
\| ^{2}_{\mathcal{H}^{\otimes 2(q-r)}}= 0$.
\item
For every $ r=1,\ldots,q-1, \; $ $
\displaystyle
\sum_{N\geqslant 2}
\displaystyle
\frac{1}{Nlog^{2}N}\,
\displaystyle
\sum_{l=1}^{N}
\displaystyle
\frac{1}{l} \| g_{l} \otimes _{r} g_{l} \| ^{2}_{
\mathcal{H}^{\otimes 2(q-r)}}< \infty $.
\item
$
\displaystyle
\sum_{N\geqslant 2}
\displaystyle
\frac{1}{Nlog^{3}N}\,
\displaystyle
\sum_{i,j=1}^{N}
\displaystyle
\frac{\llvert\mathbb{E}(G_{i}G_{j})\rrvert}{ij}<\infty $.
\end{enumerate}
Then, the sequence $(G_{N})_{N \ge 1}$ satisfies an ASCLT.\index{almost sure central limit theorem (ASCLT)}
\end{thm}

We finish this section with the following useful reduction lemma. For
its proof see Lemma $2.2$ in \cite{Azmoodeh-Nourdin}.

\begin{lemma}
\label{l9}
Consider a real-valued sequence $(a_{n})_{n \ge 1}$  converging to
$a_{\infty }\neq 0$. Consider also a sequence of
real valued random variables $(G_{n})_{n \ge 1}$. Then the sequence $(G_{n})_{n \ge 1}$
satisfies an ASCLT\index{almost sure central limit theorem (ASCLT)} if, and only if, $(a_{n}G_{n})_{n \ge 1}$ does.
\end{lemma}

\subsection{Limiting behavior of the re-normalized quadratic
variation of the spatial solution process}%
\label{sec3.2}

For fixed $t \in (0, T]$, we denote by $\mathcal{H}$ the canonical
Hilbert space associated to the Gaussian process\index{Gaussian process} $  ( u(t,x)   )_{x \in [0,1]}$ being a mild solution\index{mild solution} to Equation
(\ref{e:2g}). This Hilbert space is defined as the closure of the linear
span\index{linear span} generated by the indicator functions $
\mathbf{1}_{[0,x]},\;x>0$,\index{indicator functions} with respect to the inner product
%
\begin{equation}
\mathbb{E} \bigl( u(t,x)u(t,y) \bigr) = {<}\mathbf{1}_{[0,x]},
\mathbf{1} _{[0,y]}{>}_{\mathcal{H}} .
\end{equation}
We also denote by $I_{q}$, $ q \ge 1$, the multiple stochastic integral
with respect to the Gaussian process\index{Gaussian process} $
( u(t,x) )_{x \in [0,1]}$. So for every $x < y$ we have
\begin{equation*}
u(t,y) - u(t,x) = I_{1} ( \mathbf{1}_{[x,y]} ) .
\end{equation*}

We start our study of the limit behavior in distribution of the sequence
$(\tilde{V}_{N})_{N \ge 1}$ by the following main theorem.

\begin{thm}
\label{V}
Let $u$ be the mild solution\index{mild solution} to Equation \textup{(\ref{e:2g})}, $G$ be the
fundamental solution\index{fundamental solution} associated to the operator ${\mathcal{L}}$ and
$\tilde{V}_{N}$ be given by \textup{(\ref{var})}. If $G$ satisfies
Conditions $H_{1}([0,1])$ and $H_{3}([0,1])$, then
\begin{equation*}
\lim_{N \rightarrow \infty } \mathbb{E} \bigl( \tilde{V}_{N}^{2}
\bigr) = 1.
\end{equation*}
\end{thm}

\begin{proof}
By using Formula (\ref{prodMI}), we can write
\begin{eqnarray*}
V_{N} &=& \sum_{j=0}^{N-1}
\biggl[ \frac{  (u(t,x_{j+1} )-u(t,{x_{j}} )  )^{2} }{\mathbb{E}  (u(t,x
_{j+1} )-u(t,x_{j} )  )^{2}}-1 \biggr]
\\
&=& \sum_{j=0}^{N-1} \biggl[
\frac{{ \mathrm{I}_{1}}^{2}  (\mathbf{1}_{[x_{j},x_{j+1}]}  )}{
\mathbb{E}  (u(t,x_{j+1} )-u(t,x_{j} )  )^{2}}-1 \biggr]
\\
&=& \sum_{j=0}^{N-1}
\frac{{\mathrm{I}_{2}}  ( \mathbf{1}_{[x_{j},x
_{j+1}]}^{\otimes ^{2}}   ) }{\mathbb{E}  (u(t,x_{j+1} )-u(t,x
_{j} )  )^{2}} .
\end{eqnarray*}

By virtue of the isometry formula\index{isometry} (\ref{Isom}), we get
\begin{eqnarray*}
\mathbb{E} \bigl(V_{N}^{2} \bigr)&=&
\mathbb{E} \Biggl( \sum_{j=0}^{N-1}
\frac{{\mathrm{I}_{2}}  ( \mathbf{1}_{[x_{j},x
_{j+1}]}^{\otimes ^{2}}   ) }{\mathbb{E}  (u(t,x_{j+1} )-u(t,x
_{j} )  )^{2}
} \Biggr)^{2}
\\
&=& \sum_{j,k=0}^{N-1}
\frac{\mathbb{E}   (
{\mathrm{I}_{2}}  ( \mathbf{1}_{[x_{j},x
_{j+1}]}^{\otimes ^{2}}   )\, {\mathrm{I}_{2}}
  ( \mathbf{1}
_{[x_{k},x_{k+1}]}^{\otimes ^{2}}   )   ) }{\mathbb{E}  (u(t,x
_{j+1} )-u(t,x_{j} )  )^{2}\;\mathbb{E}  (u(t,x_{k+1} )-u(t,x_{k}
)  )^{2}}
\\
&=& 2 \sum_{j,k=0}^{N-1}
\frac{ {{<} \mathbf{1}_{[x_{j},x_{j+1}]}, \mathbf{1}_{[x_{k},x_{k+1}]} {>}
_{\mathcal{H}}^{2}} }{\mathbb{E}  (u(t,x_{j+1} )-u(t,x_{j} )  )^{2}
\;\mathbb{E}  (u(t,x_{k+1} )-u(t,x_{k} )  )^{2}}.
\end{eqnarray*}
Thus,
\begin{equation*}
\mathbb{E} \bigl({V_{N}}^{2} \bigr)=
T_{1,N}+T_{2,N} ,
\end{equation*}
where
%
\begin{equation}
\label{eq:t1n} T_{1,N}= 2\sum_{j=0}^{N-1}
\frac{ {{<} \mathbf{1}_{[x_{j},x_{j+1}]},
\mathbf{1}_{[x_{j},x_{j+1}]} {>} _{\mathcal{H}}^{2}} }{  [\mathbb{E}
  (u(t,x_{j+1} )-u(t,x_{j} )  )^{2}  ]^{2}}
\end{equation}
and
%
\begin{equation}
\label{eq:t2n} T_{2,N}= 2\sum_{j,k=0; j\neq k}^{N-1}
\frac{ {{<} \mathbf{1}_{[x_{j},x
_{j+1}]}, \mathbf{1}_{[x_{k},x_{k+1}]} {>} _{\mathcal{H}}^{2}} }{
\mathbb{E}  (u(t,x_{j+1} )-u(t,x_{j} )  )^{2}\;\mathbb{E}  (u(t,x
_{k+1} )-u(t,x_{k} )  )^{2}}.
\end{equation}
On the one hand we clearly have $T_{1,N}=2N$. On the other hand, since
Conditions $H_{1}([0,1])$ and $H_{3}([0,1])$ are satisfied, by
virtue of Lemma \ref{th:spatial-reg} we get
%
\begin{equation}
\label{T2} T_{2,N} \leqslant 2\,N^{2}\sum
_{j,k=0; j\neq k}^{N-1} { {{<} \mathbf{1} _{[x_{j},x_{j+1}]},
\mathbf{1}_{[x_{k},x_{k+1}]}{>} _{\mathcal{H}}^{2}} } \leqslant C
\,N^{2}\,\sum_{j,k=0; j\neq k}^{N-1} {
\biggl(\frac{1}{N
^{2}} \biggr)^{2}} \le C,
\end{equation}
where $C$ denotes a universal positive constant. \xch{Thus, we}{Thus,we} deduce that the
dominant term for $\mathbb{E}({\tilde{V}^{2}_{N}}) $ is obviously
$T_{1,N}$. Consequently, we obtain, for a fixed $t\in (0,T]$,
\begin{equation*}
\mathbb{E} \bigl({\tilde{V}^{2}_{N}} \bigr)=
\frac{1}{2N}\;\mathbb{E} \bigl({V^{2}_{N}} \bigr)
\longrightarrow 1 \quad \mbox{as} \ {N\longrightarrow\xch{ \infty } { . \infty
}}.\qedhere
\end{equation*}
\end{proof}

In the following theorem we establish the convergence in law of the
sequence $(\tilde{V}_{N})_{N}$.

\begin{thm}
\label{T:principal}
Consider the sequence of random variables $\tilde{V}_{N}$ defined in
\textup{(\ref{var})}. If $G$ satisfies Conditions $H_{1}([0,1])$ and
$H_{3}([0,1])$, then
\begin{equation*}
\tilde{V}_{N} \; \overset{{Law} } {\longrightarrow } \;
\mathcal{N}(0,1).
\end{equation*}
Moreover, if we denote by $d$ one of the metrics on the space of
probability measures on ${\mathbb{R}}$, including the Kolmogorov,\index{Kolmogorov}
Wasserstein\index{Wasserstein} and Total Variation measures, then for $N$ large enough
\begin{equation*}
d \bigl( {\tilde{V}}_{N},\mathcal{N}(0,1) \bigr)\leqslant
\frac{C}{
\sqrt{N}} .
\end{equation*}
\end{thm}
\begin{proof}
By virtue of Formula (\ref{DM}) we get
\begin{equation*}
D\tilde{V}_{N}=\frac{1}{\sqrt{2N}\,}\sum
_{j=0}^{N-1}\frac{{\mathrm{I}
_{1}}
  ( \mathbf{1}_{[x_{j},x_{j+1}]}   )\;\mathbf{1}_{[x
_{j},x_{j+1}]} }{\mathbb{E}  (u(t,x_{j+1} )-u(t,x_{j} )  )^{2}}.
\end{equation*}
Hence, for every fixed $t \in [0,T]$, using Formula (\ref{prodMI}), we
get
\begin{eqnarray*}
 \llVert D\tilde{V}_{N} \rrVert ^{2}_{\mathcal{H}}
&=& \frac{2}{N}\sum_{j,k=0}^{N-1}
\frac{\mathrm{I}_{2}
  ( \mathbf{1}
_{[x_{j},x_{j+1}]} \otimes \mathbf{1}_{[x_{k},x_{k+1}]}   ) \, {<}
\mathbf{1}_{[x_{j},x_{j+1}]},\mathbf{1}_{[x_{k},x_{k+1}]}{>}_{
\mathcal{H}} }{\mathbb{E}  (u(t,x_{j+1} )-u(t,x_{j} )  )^{2}\;
\mathbb{E}  (u(t,x_{k+1} )-u(t,x_{k} )  )^{2}
}
\\
&& {} +\mathbb{E} \bigl( \llVert D\tilde{V}_{N} \rrVert
^{2}_{\mathcal{H}} \bigr),
\end{eqnarray*}
and consequently,
\begin{eqnarray*}
&& {\mathbf{Var}} \bigl( \llVert D\tilde{V}_{N} \rrVert
^{2}_{\mathcal{H}} \bigr)
\\
&=& {\mathbb{E}} \bigl[ \llVert D\tilde{V}_{N} \rrVert
^{2}_{\mathcal{H}} - {\mathbb{E}} \bigl( \llVert D
\tilde{V}_{N} \rrVert ^{2}_{\mathcal{H}} \bigr)
\bigr]^{2}
\\
&=& {\mathbb{E}} \Biggl[ \frac{2}{N}\sum_{j,k=0}^{N-1}
\frac{\mathrm{I}_{2}
  ( \mathbf{1}_{[x_{j},x_{j+1}]}
\otimes \mathbf{1}_{[x_{k},x_{k+1}]}   ) \, {<} \mathbf{1}_{[x_{j},x
_{j+1}]},\mathbf{1}_{[x_{k},x_{k+1}]}{>}_{\mathcal{H}} }{\mathbb{E}
  (u(t,x_{j+1} )-u(t,x_{j} )  )^{2}\; \mathbb{E}  (u(t,x_{k+1} )-u(t,x
_{k} )  )^{2}
} \Biggr]^{2}
\\
&=& \frac{8}{N^{2}}\sum_{j,k,m,l=0}^{N-1}
\frac{\mathbb{E}  (\mathrm{I}
_{2}  ( \mathbf{1}_{[x_{j},x_{j+1}]} \otimes \mathbf{1}_{[x_{k},x
_{k+1}]}   )\,\mathrm{I}_{2}  ( \mathbf{1}_{[x_{m},x_{m+1}]}
\otimes \mathbf{1}_{[x_{l},x_{l+1}]}   )   ) }{\mathbb{E}
  (u(t,x_{j+1} )-u(t,x_{j} )  )^{2}\; \mathbb{E}  (u(t,x_{k+1} )-u(t,x
_{k} )  )^{2}}
\\
&&{} \times \frac{{<}\mathbf{1}_{[x_{j},x_{j+1}]},\mathbf{1}_{[x_{k},x_{k+1}]}{>}_{
\mathcal{H}} \, {<} \mathbf{1}_{[x_{m},x_{m+1}]},\mathbf{1}_{[x_{l},x
_{l+1}]} {>}_{\mathcal{H}} }{\mathbb{E}  (u(t,x_{m+1} )-u(t,x_{m} )
  )^{2}\; \mathbb{E}  (u(t,x_{l+1} )-u(t,x_{l} )  )^{2}}
\\
&=& \frac{8}{N^{2}}\sum_{j,k,m,l=0}^{N-1}
\frac{{<}\mathbf{1}_{[x_{j},x_{j+1}]}
\tilde{\otimes }\mathbf{1}_{[x_{k},x_{k+1}]} , \mathbf{1}_{[x_{m},x
_{m+1}]} \tilde{\otimes }\mathbf{1}_{[x_{l},x_{l+1}]} {>}_{\mathcal{H}
^{\otimes ^{2} }}}{  \llVert  \mathbf{1}_{[x_{j},x_{j+1}]}
 \rrVert  _{\mathcal{H}}^{2}\;  \llVert  \mathbf{1}_{[x_{k},x_{k+1}]}
 \rrVert  _{\mathcal{H}}^{2}}
\\
&&{} \times \frac{
{<} \mathbf{1}_{[x_{j},x_{j+1}]},\mathbf{1}_{[x_{k},x_{k+1}]}{>}_{
\mathcal{H}} \, {<} \mathbf{1}_{[x_{m},x_{m+1}]},\mathbf{1}_{[x_{l},x
_{l+1}]} {>}_{\mathcal{H}} }{ \llVert  \mathbf{1}_{[x_{m},x_{m+1}]}
 \rrVert  _{\mathcal{H}}^{2} \;  \llVert  \mathbf{1}_{[x_{l},x_{l+1}]}
 \rrVert  _{\mathcal{H}}^{2}},
\end{eqnarray*}
where $f\tilde{\otimes } g$ denotes the symmetrization of the tensor
product $f\otimes g$ that satisfies
\begin{equation*}
f\tilde{\otimes } g= \frac{1}{2}(f\otimes g+g\otimes f)
\end{equation*}
and
\begin{equation*}
{<}f\tilde{\otimes } g, f'\tilde{\otimes } g'
{>}_{\mathcal{H}}= \frac{1}{2} \bigl( {<}f,f'{>}_{\mathcal{H}}{<}g,g'{>}_{\mathcal{H}}+
{<}f,g'{>}_{
\mathcal{H}}{<}g,f'{>}_{\mathcal{H}}
\bigr) .
\end{equation*}
Therefore,
\begin{eqnarray*}
&& {\mathbf{Var}} \bigl( \llVert D\tilde{V}_{N} \rrVert
^{2}_{\mathcal{H}} \bigr)
\\
&=& \frac{8}{N^{2}} \sum_{j,k,m,l=0}^{N-1}
\frac{{<}\mathbf{1}_{[x_{j},x_{j+1}]} ,
\mathbf{1}_{[x_{k},x_{k+1}]} {>}_{\mathcal{H}} \,{<} \mathbf{1}_{[x_{m},x
_{m+1}]} ,\mathbf{1}_{[x_{l},x_{l+1}]} {>}_{\mathcal{H}}}{  \llVert
\mathbf{1}_{[x_{j},x_{j+1}]}  \rrVert  _{\mathcal{H}}^{2}\;  \llVert
\mathbf{1}_{[x_{k},x_{k+1}]}  \rrVert  _{\mathcal{H}}^{2}
}
\\
&&{} \times \frac{{<} \mathbf{1}_{[x_{j},x_{j+1}]},\mathbf{1}_{[x_{k},x_{k+1}]}{>}_{
\mathcal{H}} \, {<} \mathbf{1}_{[x_{m},x_{m+1}]},\mathbf{1}_{[x_{l},x
_{l+1}]} {>}_{\mathcal{H}} }{ \llVert  \mathbf{1}_{[x_{m},x_{m+1}]}
 \rrVert  _{\mathcal{H}}^{2} \;  \llVert  \mathbf{1}_{[x_{l},x_{l+1}]}
 \rrVert  _{\mathcal{H}}^{2}}
\\
&= &D_{4,N}+ D_{3,N}+ D_{2,N}+
D_{1,N}
\end{eqnarray*}
where $D_{i,N}$, for every $i\in \{1,2,3,4\}$, contains all the terms
with $i$ equal indices. So, $D_{4,N}$ contains all the summands above
with $j= k = m = l$; that is
\begin{equation*}
D_{4,N} = \frac{8}{N^{2}} \sum_{j= 0}^{N-1}
1 = \frac{8}{N}.
\end{equation*}
As for $D_{3,N}$, it contains all the terms corresponding to
$j= k = l \neq m$; so, since $G$ satisfies Conditions $H_{1}([0,1])$ and
$H_{3}([0,1])$, using Lemma \ref{th:spatial-reg} we get
\begin{eqnarray*}
D_{3,N} &\le & \frac{8}{N^{2}} \sum
_{l, m=0}^{N-1} \frac{  \llVert  \mathbf{1}_{[x_{l},x_{l+1}]}
 \rrVert  _{\mathcal{H}}^{4}
{<}\mathbf{1}_{[x_{l},x_{l+1}]} ,\mathbf{1}
_{[x_{m},x_{m+1}]} {>}_{\mathcal{H}}^{2}}{  \llVert  \mathbf{1}_{[x_{l},x
_{l+1}]}  \rrVert  _{\mathcal{H}}^{6}\;  \llVert  \mathbf{1}_{[x_{m},x
_{m+1}]}  \rrVert  _{\mathcal{H}}^{2}}
\\
&\le & \frac{C}{N^{2}} \sum_{l, m=0}^{N-1}
\frac{(\frac{1}{N^{2}})^{2}}{(
\frac{1}{N})^{2}} = \frac{C}{N^{2}}.
\end{eqnarray*}
By the same way, and using again Lemma \ref{th:spatial-reg}, we show
that
\begin{equation*}
D_{2,N}\leqslant \frac{C}{N^{2}} \quad \mbox{and} \quad
D_{1,N} \leqslant \frac{C}{N^{2}}.
\end{equation*}
All this allow us to get
\begin{equation*}
{\mathbf{Var}} \bigl( \llVert D \tilde{V}_{N} \rrVert
^{2}_{\mathcal{H}} \bigr) \leqslant \frac{C}{N}.
\end{equation*}
Moreover, we have
\begin{equation*}
\mathbb{E} \bigl( \llVert D\tilde{V}_{N} \rrVert
^{2}_{\mathcal{H}} \bigr)=2 \mathbb{E}(\tilde{V}_{N})^{2}=
\frac{\mathbb{E}(V_{N}^{2})}{N}= \frac{1}{N}(T_{1,N}+T_{2,N})=
2+ \frac{T_{2,N}}{N}
\end{equation*}
where $T_{1,N}$ and $T_{2,N}$ are defined by (\ref{eq:t1n}) and
(\ref{eq:t2n}). This and Inequality (\ref{T2}) allow us to deduce that
\begin{equation*}
\mathbb{E} \bigl( \llVert D\tilde{V}_{N} \rrVert
^{2}_{\mathcal{H}} \bigr)-2\leqslant \displaystyle
\frac{C}{N}.
\end{equation*}
By virtue of Theorem \ref{theo-inter}, the proof of Theorem
\ref{T:principal} is completed.
\end{proof}

\subsection{Almost sure central limit\index{almost sure central limit theorem (ASCLT)} theorem}%
\label{sec3.3}
The following theorem is 
a kind of extension of Theorem
\ref{T:principal}.\vadjust{\goodbreak}
%
\begin{thm}
\label{ASCLM}
If $G$ satisfies Conditions $H_{1}([0,1])$ and $H_{3}([0,1])$, then the
sequence $(\tilde{V}_{N})_{N \ge 1}$ satisfies an ASCLT.\index{almost sure central limit theorem (ASCLT)}
\end{thm}

\begin{proof}
Denoting $\sigma _{N} =\sqrt{\mathbb{E}( \tilde{V}_{N}^{2}) }$, for every
$N \ge 1$, according to Theorem \ref{V}, we have $\lim_{N \rightarrow \infty } \sigma _{N} = 1$. So without loss of
generality, we assume that $\inf_{N \ge 1} \sigma _{N} = \sigma _{0} > 0$ and we consider
$G_{N} = \frac{\tilde{V}_{N}}{\sigma _{N}}$, for every $N \ge 1$.

According to Lemma \ref{l9}, to obtain Theorem \ref{ASCLM} it suffices
to show that the sequence $(G_{N})_{N \ge 1}$ satisfies an ASCLT.\index{almost sure central limit theorem (ASCLT)} To
this end, since for every $N \ge 1$ we have $G_{N}= \mathrm{I}_{2}(g_{N})$ with
\begin{equation*}
g_{N}:=\frac{1}{ \sigma _{N}\sqrt{2N}}\sum_{j=0}^{N-1}
\frac{
\mathbf{1}_{[x_{j},x_{j+1}]}^{\otimes ^{2}} }{\mathbb{E}  (u(t,x_{j+1}
)-u(t,x_{j} )  )^{2}},
\end{equation*}
and since we obviously have ${\mathbb{E}}(G_{N}^{2}) = 1$, for every $N \ge 1$, it suffices to check
the three last assumptions in Theorem \ref{T:7}.

By the $1$st contraction defined by (\ref{contraction}), we obtain
\begin{eqnarray*}
g_{l}\otimes _{1} g_{l}&=&
\frac{1}{ 2\sigma _{l}^{2} l} \sum_{j,k=0}^{l-1}
\displaystyle \frac{\mathbf{1}_{[x_{j},x_{j+1}]}^{\otimes ^{2}} \otimes _{1}
\mathbf{1}_{[x_{k},x_{k+1}]}^{\otimes ^{2}} }{\mathbb{E}  (u(t,x_{j+1}
)-u(t,x_{j} )  )^{2}\,\mathbb{E}  (u(t,x_{k+1} )-u(t,x_{k} )
  )^{2} }
\\
&= & \frac{1}{ 2 \sigma _{l}^{2} l} \sum_{j,k=0}^{l-1}
\frac{ {<}\mathbf{1}_{[x_{j},x_{j+1}]} , \mathbf{1}_{[x_{k},x_{k+1}]} {>}_{
\mathcal{H}} \mathbf{1}_{[x_{j},x_{j+1}]} \otimes \mathbf{1}_{[x_{k},x
_{k+1}]} }{\mathbb{E}  (u(t,x_{j+1} )-u(t,x_{j} )  )^{2}\,
\mathbb{E}  (u(t,x_{k+1} )-u(t,x_{k} )  )^{2} }.
\end{eqnarray*}
Therefore,
\begin{eqnarray*}
&& \llVert g_{l}\otimes _{1}
g_{l} \rrVert ^{2}_{\mathcal{H}^{\otimes
^{2}}}
\\
&=& \frac{1}{ 4\sigma _{l}^{4}l^{2}}\sum_{j,k,m,p=0}^{l-1}
\frac{ {<}
\mathbf{1}_{[x_{j},x_{j+1}]},\mathbf{1}_{[x_{k},x_{k+1}]} {>}_{
\mathcal{H}} \, {<}\mathbf{1}_{[x_{m},x_{m+1}]},\mathbf{1}_{[x_{p},x
_{p+1}]} {>}_{\mathcal{H}}}{\mathbb{E}  (u(t,x_{j+1} )-u(t,x_{j} )
  )^{2}\,\mathbb{E}  (u(t,x_{k+1} )-u(t,x_{k} )  )^{2}}
\\
& & {}\times \frac{ {<}\mathbf{1}_{[x_{j},x_{j+1}]} \tilde{\otimes }
\mathbf{1}_{[x_{k},x_{k+1}]},
\mathbf{1}_{[x_{m},x_{m+1}]}
\tilde{ \otimes }\mathbf{1}_{[x_{p},x_{p+1}]} {>}_{\mathcal{H}} }{
\mathbb{E}  (u(t,x_{m+1} )-u(t,x_{m} )  )^{2}\,\mathbb{E}  (u(t,x
_{p+1} )-u(t,x_{p} )  )^{2} }
\\
&=& \frac{1}{ 4\sigma _{l}^{2} l^{2}} \sum_{j,k,m,p=0}^{l-1}
\frac{ {<}
\mathbf{1}_{[x_{j},x_{j+1}]} ,
\mathbf{1}_{[x_{k},x_{k+1}]} {>}_{
\mathcal{H}}\, {<}\mathbf{1}_{[x_{m},x_{m+1}]} ,
\mathbf{1}_{[x_{p},x
_{p+1}]} {>}_{\mathcal{H}} }{\mathbb{E}  (u(t,x_{j+1} )-u(t,x_{j} )
  )^{2}\,\mathbb{E}  (u(t,x_{k+1} )-u(t,x_{k} )  )^{2} }
\\
&& {} \times \frac{ {<}\mathbf{1}_{[x_{j},x_{j+1}]},\mathbf{1}_{[x_{m},x_{m+1}]}
{>}_{\mathcal{H}}
\, {<}\mathbf{1}_{[x_{k},x_{k+1}]} ,\mathbf{1}_{[x_{p},x
_{p+1}]} {>}_{\mathcal{H}} }{\mathbb{E}  (u(t,x_{m+1} )-u(t,x_{m} )
  )^{2}\,\mathbb{E}  (u(t,x_{p+1} )-u(t,x_{p} )  )^{2} } .
\end{eqnarray*}
Since $\frac{1}{\sigma _{l}^{4}} \le \frac{1}{\sigma _{0}^{4}} $ for every
$l \ge 1$, and since $G$ satisfies Conditions $H_{1}([0,1])$ and
$H_{3}([0,1])$, proceeding in the same way as in the proof of Theorem
\ref{T:principal}, we get
%
\begin{equation}
\label{e:27} \llVert g_{l}\otimes _{1}
g_{l} \rrVert ^{2}_{\mathcal{H}^{\otimes
^{2}}}\leqslant
\frac{C}{l},
\end{equation}
and consequently the second assumption in Theorem \ref{T:7} is satisfied.\newpage

From Inequality (\ref{e:27}) we also deduce that
\begin{eqnarray*}
\sum_{N\geqslant 2}\frac{1}{Nlog^{2}N}\,\sum
_{l=1}^{N}\frac{1}{l} \llVert
g_{l} \otimes _{1} g_{l} \rrVert
^{2}_{\mathcal{H}^{\otimes
2}} &\leqslant & C \,\sum
_{N\geqslant 2}\frac{1}{Nlog^{2}N}\,\sum
_{l=1}^{\infty }\frac{1}{l
^{2}}
\\
&\leqslant & C \, \sum_{N\geqslant 2}\frac{1}{Nlog^{2}N}
< \infty ,
\end{eqnarray*}
that means that the third assumption in Theorem \ref{T:7} is also
satisfied.

Let us now check the last assumption in Theorem \ref{T:7}. Since we have
\begin{equation*}
\mathbb{E}(G_{i} G_{j} )= 2 {<}g_{i},g_{j}
{>}_{\mathcal{H}^{\otimes ^{2}}}
\end{equation*}
and since Conditions $H_{1}([0,1])$ and $H_{3}([0,1])$ are satisfied,
using Lemma \ref{th:spatial-reg} we get:
\begin{equation*}
\mbox{If}\ i=j, \quad {<}g_{i},g_{i}{>}_{\mathcal{H}^{\otimes ^{2}}}
\leqslant \displaystyle \frac{C}{i} \quad \mbox{and} \quad \mbox{if}\
i>j, \quad {<}g_{i},g_{j} {>}_{\mathcal{H}^{\otimes ^{2}}}
\leqslant C\, \sqrt{ \displaystyle \frac{j}{i}}.
\end{equation*}
Therefore,
\begin{eqnarray*}
&& \sum_{N\geqslant 2}\frac{1}{Nlog^{3}N}\,
\sum_{i,j=1}^{N} \frac{ \llvert  \mathbb{E}(G_{i}G_{j}) \rrvert  }{ij}
\\
&=& \sum_{N\geqslant 2}\frac{1}{Nlog^{3}N}\, \Biggl[
\sum_{i\neq j=1}^{N} \frac{ \llvert  \mathbb{E}(G_{i}G_{j}) \rrvert  }{ij}+
\sum_{i=1}^{N} \frac{ \llvert  \mathbb{E}(G_{i}^{2}) \rrvert  }{i^{2}}
\Biggr]
\\
&\leqslant & 2\sum_{N\geqslant 2}\frac{1}{Nlog^{3}N}\,
\Biggl[ 2 \sum_{i > j=1}^{N}
\frac{  \llvert  {<}g_{i},g_{j} {>}_{\mathcal{H}^{\otimes
^{2}}} \rrvert  }{ij}+\sum_{i=1}^{N}
\frac{ \llvert  {<}g_{i},g_{i} {>}_{
\mathcal{H}^{\otimes ^{2}}}  \rrvert  }{i^{2}} \Biggr]
\\
&\leqslant & C\, \sum_{N\geqslant 2} \frac{2}{Nlog^{3}N}
\, \Biggl[ \sum_{i > j=1}^{N}
\frac{ 2}{i\sqrt{ij}}+ \sum_{i=1}^{N}
\frac{1 }{i^{3}} \Biggr]
\\
&<& \infty .\hspace*{280pt}\qedhere
\end{eqnarray*}
\end{proof}

\section{Stochastic heat equation\index{stochastic heat equation} with piecewise constant coefficients}%
\label{sec4}
The study done in the previous section allows us to make a new step in
the investigation of the solution to the SPDE\index{stochastic partial differential equation (SPDE)} (\ref{e:2p}). Equation
(\ref{e:2p}) is obviously a particular case of (\ref{e:2g}). Indeed, the
operator ${\mathcal{L}}_{p}$ defined by (\ref{e:1.2}) can be written in
the form (\ref{e:1.1}) with
\begin{equation*}
r(x) = 2 \rho (x) \; \mathrm{and} \; R(x):= \rho (x) A(x).
\end{equation*}

In the following proposition we present the expression of the
fundamental solution\index{fundamental solution} associated to the operator ${\mathcal{L}}_{p}$. For
a proof see, e.g., \cite{CZ,Zi,Zil} and
\cite{MZ0}.

\begin{proposition}
\label{pro:sol-fund}
There exists a unique fundamental solution $G(t-s, x, y)$ associated to
the operator ${\mathcal{L}}_{p}$. It can be explicitly expressed as
%
\begin{equation}
\label{e:5} G(u,x,z)=m(u)\; \biggl[\frac{1}{\sqrt{a_{1}}}
\,A^{-}(u,x,z)1_{\{z
\leqslant 0\}}+\frac{1}{\sqrt{a_{2}}}
\,A^{+}(u,x,z)1_{\{z>0\}} \biggr]
\end{equation}
with
%
\begin{equation}
m(u)= \frac{1}{\sqrt{2\pi u}}\,1_{\{u>0\}},
\end{equation}
%
\begin{equation}
\label{eq:A} \lleft \{ %
\begin{array}{@{}l}
\displaystyle
A^{-}(u,x,z)=E^{-}(u,x,z)-\beta E^{+}(u,x,z),
\\
\noalign{\vskip2mm}
\displaystyle
A^{+}(u,x,z)=E^{-}(u,x,z)+\beta E^{+}(u,x,z),
\end{array} %
\rright .
\end{equation}
%
\begin{equation}
\label{E} \lleft \{ %
\begin{array}{@{}l}
\displaystyle
E^{-}(u,x,z)=\exp    \lleft(-\frac{(f(z)-f(x))^{2}}{2u}   \rright),
\\
\noalign{\vskip2mm}
\displaystyle
E^{+}(u,x,z)=\exp   \lleft (-\frac{( \llvert f(z) \rrvert + \llvert f(x) \rrvert )^{2}}{2u}  \rright ),
\end{array} %
\rright .
\end{equation}
%
\begin{equation}
\label{eq:f} f(z)= \frac{z}{\sqrt{a_{1}}}{\mathbf{1}}_{\{z\leqslant 0\}}+
\frac{z}{\sqrt{a
_{2}}} {\mathbf{1}}_{\{z>0\}} \quad \mathit{and}\quad
\beta = \displaystyle \frac{\rho _{2} \sqrt{a_{2}}-
\rho _{1}\sqrt{a_{1}}}{\rho _{2} \sqrt{a
_{2}}+
\rho _{1}\sqrt{a_{1}}} .
\end{equation}
\end{proposition}

In this section, by making an in-depth study of the terms $f, A^{-}$ and
$A^{+}$ defined in Expressions (\ref{eq:A}) and (\ref{eq:f}), we will
prove the following theorem.

\begin{thm}
\label{Teo.part}
Let u be the mild solution\index{mild solution} to Equation \textup{(\ref{e:2p})} and $\tilde{V}
_{N}$ be the sequence given by \textup{(\ref{var})}. Suppose that the
coefficients $A$ and $\rho $ defined in \textup{(\ref{eq:coefA})} satisfy
%
\begin{equation}
\label{CrhoA} \max \biggl( 1 , \frac{\sqrt{a_{1}}}{\sqrt{a_{2}}} \biggr) \le
\frac{
\rho _{2}}{\rho _{1}} .
\end{equation}
Then the following is valid:
\begin{enumerate}%
\item
\begin{equation*}
\tilde{V}_{N} \; \overset{{Law} } {\longrightarrow } \;
\mathcal{N}(0,1).
\end{equation*}
Moreover, if we denote by $d$ one of the metrics on the space of
probability measures on ${\mathbb{R}}$, including the Kolmogorov,\index{Kolmogorov}
Wasserstein\index{Wasserstein} and Total Variation measures, then for $N$ large enough
\begin{equation*}
d \bigl( {\tilde{V}}_{N},\mathcal{N}(0,1) \bigr)\leqslant
\frac{C}{
\sqrt{N}}.
\end{equation*}%
\item
The sequence $
\displaystyle
( {\tilde{V}}_{N})_{N \ge 1}$ satisfies an ASCLT.\index{almost sure central limit theorem (ASCLT)}
\end{enumerate}
\end{thm}

\begin{remark}
If $a_{1}= a_{2} = 1$ and $\rho _{1}= \rho _{2} = 2$, then Condition
(\ref{CrhoA}) is well satisfied. Thus, the result of Theorem
\ref{Teo.part} applies to the standard stochastic heat equation\index{stochastic heat equation} with
the time-space white noise and it corresponds exactly to that obtained in
\cite{TT}.
\end{remark}

To prove Theorem \ref{Teo.part}, we shall first establish the following
lemmas.

\subsection{Preliminary lemmas}%
\label{sec4.1}
%
\begin{lemma}
\label{Majf}
Consider $f$, the function defined in \textup{(\ref{eq:f})}. For every
$x, y \in {\mathbb{R}}$,
\begin{equation*}
\min \biggl( \frac{1}{\sqrt{a_{2}}}, \frac{1}{\sqrt{a_{1}}} \biggr) \llvert y-x
\rrvert \le \bigl\llvert f(y) - f(x) \bigr\rrvert \le \max \biggl(
\frac{1}{\sqrt{a_{2}}}, \frac{1}{\sqrt{a_{1}}} \biggr) \llvert y-x \rrvert .
\end{equation*}
\end{lemma}

\begin{proof}
Expression (\ref{eq:f}) allows to get
%
\begin{equation}
\label{eqf} f(y) - f(x) = \lleft \{ %
\begin{array}{@{}r@{\quad}c@{\ }l}
\displaystyle
\frac{y-x}{\sqrt{a_{2}}} &\mbox{if}& y > 0 \; x > 0,
\\
\noalign{\vskip2mm}
\displaystyle
\frac{y-x}{\sqrt{a_{1}}} &\mbox{if}& y \le 0 \; x \le 0,
\\
\noalign{\vskip2mm}
\displaystyle
\frac{y}{\sqrt{a_{1}}} - \frac{x}{\sqrt{a_{2}}} &\mbox{if}& y \le 0 \; x
> 0,
\\
\noalign{\vskip2mm}
\displaystyle
\frac{y}{\sqrt{a_{2}}} - \frac{x}{\sqrt{a_{1}}} &\mbox{if}& y > 0 \; x
\le 0.
\end{array} %
\rright .
\end{equation}
If $xy \ge 0$, both inequalities in Lemma \ref{Majf} are directly
obtained from (\ref{eqf}). If $y > 0$ and $x< 0$,
\begin{eqnarray*}
&& \max \biggl( \frac{1}{\sqrt{a_{2}}}, \frac{1}{\sqrt{a_{1}}} \biggr)
\llvert y-x \rrvert - \bigl\llvert f(y) - f(x) \bigr\rrvert
\\
&=& \max \biggl( \frac{1}{\sqrt{a_{2}}}, \frac{1}{\sqrt{a_{1}}} \biggr) (y-x ) -
\frac{y}{\sqrt{a_{2}}} + \frac{x}{\sqrt{a_{1}}}
\\
&=& y \biggl[ \max \biggl( \frac{1}{\sqrt{a_{2}}}, \frac{1}{\sqrt{a_{1}}} \biggr)
-\frac{1}{\sqrt{a_{2}}} \biggr] - x \biggl[ \max \biggl( \frac{1}{\sqrt{a
_{2}}},
\frac{1}{\sqrt{a_{1}}} \biggr) -\frac{1}{\sqrt{a_{1}}} \biggr]
\\
&> & 0
\end{eqnarray*}
and
\begin{eqnarray*}
&& \min \biggl( \frac{1}{\sqrt{a_{2}}}, \frac{1}{\sqrt{a_{1}}} \biggr)
\llvert y-x \rrvert - \bigl\llvert f(y) - f(x) \bigr\rrvert
\\
&=& \min \biggl( \frac{1}{\sqrt{a_{2}}}, \frac{1}{\sqrt{a_{1}}} \biggr) (y-x ) -
\frac{y}{\sqrt{a_{2}}} + \frac{x}{\sqrt{a_{1}}}
\\
&=& y \biggl[ \min \biggl( \frac{1}{\sqrt{a_{2}}}, \frac{1}{\sqrt{a_{1}}} \biggr)
-\frac{1}{\sqrt{a_{2}}} \biggr] - x \biggl[ \min \biggl( \frac{1}{\sqrt{a
_{2}}},
\frac{1}{\sqrt{a_{1}}} \biggr) -\frac{1}{\sqrt{a_{1}}} \biggr]
\\
& < & 0.
\end{eqnarray*}
The proof of both inequalities in the case where $y < 0$ and
$x > 0$ is similar.
\end{proof}

\begin{lemma}
\label{lem:UBp1}
 There exists a universal positive constant
$C$, such that
\begin{equation*}
\int_{0}^{t} \frac{1}{2 \pi u}
\displaystyle \int_{{ \mathbb{R}}} \bigl\llvert E^{+}(u,y,z)-E^{+}(u,x,z)
\bigr\rrvert ^{2} \,dz\; du \leqslant C \; \llvert y-x \rrvert
\end{equation*}
and
\begin{equation*}
\int_{0}^{t} \frac{1}{2 \pi u}
\displaystyle \int_{{ \mathbb{R}}} \bigl\llvert E^{-}(u,y,z)-E^{-}(u,x,z)
\bigr\rrvert ^{2} \,dz\; du \leqslant C \; \llvert y-x \rrvert
\end{equation*}
for every $t > 0$ and $x,y \in { \mathbb{R}}$.
\end{lemma}
\begin{proof}
We present only the proof of the first inequality; the proof of the
second one is similar. By using Expression (\ref{E}), we get
\begin{eqnarray*}
&& \int_{0}^{t} \frac{1}{2 \pi u} \int
_{{ \mathbb{R}}} \bigl\llvert E^{+}(u,y,z)-E^{+}(u,x,z)
\bigr\rrvert ^{2} \,dz\; du
\\
&=& \int_{0}^{t} \int_{{ \mathbb{R}}}
\biggl[ \frac{1}{\sqrt{2\pi u}} \exp \biggl(- \frac{( \llvert  f(z) \rrvert  + \llvert  f(y) \rrvert  )^{2}}{2u} \biggr)
\\
&& {} - \frac{1}{\sqrt{2 \pi u}} \exp \biggl(- \frac{( \llvert  f(z)
 \rrvert  + \llvert  f(x) \rrvert  )^{2}}{2u} \biggr)
\biggr]^{2}\,dz\;du
\\
&=& \int_{0}^{t} \int_{0}^{\infty }
\biggl[ \frac{1}{\sqrt{2\pi u}} \exp \biggl(\frac{(z/\sqrt{a_{2}}+ \llvert  f(y) \rrvert  )^{2}}{2u} \biggr)
\\
&& {} - \frac{1}{\sqrt{2 \pi u}} \exp \biggl(- \frac{(z/\sqrt{a
_{2}}+ \llvert  f(x) \rrvert  )^{2}}{2u} \biggr)
\biggr]^{2}\,dz\;du
\\
&&{}+\int_{0}^{t} \int_{-\infty }^{0}
\biggl[ \frac{1}{\sqrt{2\pi u}} \exp \biggl(-\frac{(-z/\sqrt{a_{1}}+ \llvert  f(y) \rrvert  )^{2}}{2u} \biggr)
\\
&& {}- \frac{1}{\sqrt{2 \pi u}} \exp \biggl(- \frac{(-z/\sqrt{a
_{1}}+ \llvert  f(x) \rrvert  )^{2}}{2u} \biggr)
\biggr]^{2}\,dz\;du.
\end{eqnarray*}
The changes of variables $Z= z/\sqrt{a_{2}}+\llvert f(x) \rrvert $ in the first integral and
$Z= -z/\sqrt{a_{1}}+\llvert f(x) \rrvert $, in the second one give
%
\begin{eqnarray}
\label{eqe+} %
&& \int_{0}^{t}
\frac{1}{2 \pi u} \int_{{ \mathbb{R}}} \bigl\llvert
E^{+}(u,y,z)-E^{+}(u,x,z) \bigr\rrvert ^{2}
\,dz\; du
\nonumber
\\
&=& \sqrt{a_{2}} \; \int_{0}^{t}
\int_{ \llvert  f(x)  \rrvert  }^{+\infty } \biggl[ \frac{1}{\sqrt{2\pi u}}
\exp \biggl(-\frac{
  (Z+( \llvert  f(y)
 \rrvert  - \llvert  f(x) \rrvert  )  )^{2}}{2u} \biggr)
\nonumber
\\
&& {}- \frac{1}{\sqrt{2 \pi u}} \exp \biggl(- \frac{Z^{2}}{2u} \biggr) \;
\biggr]^{2}\,dZ\;du
\nonumber
\\
&& {}+ \sqrt{a_{1}} \; \int_{0}^{t}
\int_{ \llvert  f(x)  \rrvert  }^{+\infty } \biggl[ \frac{1}{\sqrt{2\pi u}}
\exp \biggl(-\frac{
  (Z+( \llvert  f(y)
 \rrvert  - \llvert  f(x) \rrvert  )  )^{2}}{2u} \biggr)
\nonumber
\\
&& {} - \frac{1}{\sqrt{2 \pi u}} \exp \biggl(- \frac{Z^{2}}{2u} \biggr) \;
\biggr]^{2}\,dZ\;du.
\end{eqnarray}
Therefore,
%
\begin{eqnarray}
&& \int_{0}^{t}
\frac{1}{2 \pi u} \int_{{ \mathbb{R}}} \bigl\llvert
E^{+}(u,y,z)-E^{+}(u,x,z) \bigr\rrvert ^{2}
\,dz\; du
\nonumber
\\
&\leqslant & 2 \max ( \sqrt{a_{1}}, \sqrt{a_{2}}) \;
\int_{0}^{t} \int_{{ \mathbb{R}}}
\biggl[ \frac{1}{\sqrt{2\pi u}} \exp \biggl(-\frac{   (Z+\tilde{H}  )^{2}}{2u} \biggr)
\nonumber
\\
&& {} - \frac{1}{\sqrt{2 \pi u}} \exp \biggl(- \frac{Z^{2}}{2u} \biggr) \;
\biggr]^{2}\,dZ\;du
\nonumber
\\
&= & 2 \max ( \sqrt{a_{1}}, \sqrt{a_{2}}) \; \int
_{0}^{t} \int_{{ \mathbb{R}}} \bigl[
p_{u}(Z+\tilde{H}) - p_{u}(Z) \; \bigr]
^{2}\,dZ\;du,
\end{eqnarray}
where $\tilde{H}= \llvert f(y)\rrvert-\llvert f(x)\rrvert$ and $p_{u}$ denotes the heat kernel
defined by
%
\begin{equation}
\label{GF} p_{t}(x)= \frac{1}{\sqrt{2\pi t}} \exp \biggl(
\frac{-x^{2}}{2 t} \biggr), \quad \mbox{for every}\ t > 0\ \mbox{and}\ x\in
\mathbb{R}.
\end{equation}

It is known that the Fourier transform of $p_{u}$ is 
\begin{equation*}
\mathcal{F}(p_{t}) (\xi )= e^{- t \xi ^{2} /2} \quad \forall \,\xi
\in { \mathbb{R}},\,t>0.
\end{equation*}
By virtue of the Plancherel theorem we can write
\begin{eqnarray*}
&& \int_{0}^{t}{ds}\int
_{\mathbb{R}}{ \bigl[p_{s}(v+ h)-p_{s}(v)
\bigr]^{2}dv}
\\
&=& \frac{1}{2\pi } \int_{0}^{t}{ds}\int
_{-\infty }^{\infty }{ \bigl\llvert e^{- s \xi ^{2}/2+i\xi h}-
e^{-s
\xi ^{2}/2} \bigr\rrvert ^{2}\;d\xi }
\\
&=& \frac{1}{\pi } \int_{0}^{t}{ds}\int
_{-\infty }^{\infty }{e^{-s \xi
^{2}} \bigl(1-\cos (h
\xi ) \bigr)\; d\xi }
\end{eqnarray*}
for every $h \in {\mathbb{R}}$. Applying Fubini's Theorem and using the
fact that the functions cosine and $\xi \longmapsto \frac{1-\cos (h \xi )}{\xi ^{2}}$ are even we get:
%
\begin{eqnarray}
\label{eq:planch} %
\int_{0}^{t}{ds}
\int_{\mathbb{R}} \bigl[p_{s}(v+h)-p_{s}(v)
\bigr]^{2} dv &=& \frac{1}{\pi }\int_{-\infty }^{\infty }
\Biggl[\int_{0} ^{t}e^{-s \xi ^{2}}\;ds
\Biggr]\; \bigl(1-\cos (h \xi ) \bigr) \;d\xi
\nonumber
\\
&= & \frac{1}{ \pi }\int_{-\infty }^{\infty }
\bigl(1-e^{-t \xi ^{2}} \bigr)\frac{1-
\cos (h \xi )}{\xi ^{2}}\; d\xi
\nonumber
\\
&=& \frac{2}{ \pi }\int_{0}^{\infty }
\bigl(1-e^{-t \xi ^{2}} \bigr) \frac{1-\cos ( \llvert h \rrvert
\xi )}{\xi ^{2}}\; d\xi .  %
\end{eqnarray}

Suppose that $h \neq 0$. By a simple change of variables in
(\ref{eq:planch}), using the fact that
\begin{equation*}
\forall \theta \ge 0 \quad 1-\exp (-\theta )\leqslant 1,
\end{equation*}
we obtain
%
\begin{eqnarray}
\label{eq:39'} %
\int_{0}^{t}{ds}
\int_{\mathbb{R}} \bigl[p_{s}(v+h)-p_{s}(v)
\bigr]^{2} dv &=& \ \displaystyle \frac{2\, \llvert h \rrvert }{ \pi }\int
_{0}^{\infty } \bigl(1-e^{-t \frac{\xi
^{2}}{h^{2}}} \bigr)
\frac{1-\cos ( \xi )}{\xi ^{2}}\; d\xi
\nonumber
\\
& \leqslant & \frac{2\, \llvert h \rrvert }{ \pi }\int_{0}^{\infty }
\frac{1-\cos ( \xi )}{\xi ^{2}} \; d\xi . %
\end{eqnarray}
The function $g: \xi \longmapsto \frac{1-\cos ( \xi )}{\xi ^{2}}$ is continuous on the
interval $(0, + \infty )$ and consequently it is locally integrable. In
addition, on the one hand, $\lim_{\xi \rightarrow 0} g(\xi ) = \frac{1}{2}$, which implies that
$g$ is integrable in a neighbourhood of $0$. On the other hand,
$\llvert g(\xi )\rrvert \le \frac{2}{\xi ^{2}} $ for every $\xi > 1$ and $\int_{1}^{+\infty } \frac{1}{\xi ^{2}} d\xi < \infty $, which entails the
integrability of $g$ on a neighbourhood of $+ \infty $. From all this,
one can deduce that the integral $\int_{0}^{\infty } \frac{1-\cos ( \xi )}{\xi ^{2}}\; d\xi $ is convergent
and, consequently, by using (\ref{eq:39'}),
%
\begin{equation}
\label{eq:36} \displaystyle \int_{0}^{t}{ds}
\int_{\mathbb{R}} \bigl[p_{s}(v+h)-p_{s}(v)
\bigr]^{2} dv \leqslant C \llvert h \rrvert ,
\end{equation}
for every real $h \neq 0$. Morover, Inequality (\ref{eq:36}) is
obviously true for $h = 0$. Thus, (\ref{eq:36}) is statisfied for every
$h \in {\mathbb{R}}$.

This and Inequality (\ref{eqe+}) imply that
\begin{eqnarray*}
\int_{0}^{t} \frac{1}{2 \pi u}
\int_{{ \mathbb{R}}} \bigl\llvert E^{+}(u,y,z)-E^{+}(u,x,z)
\bigr\rrvert ^{2} \,dz\; du &\le & C \bigl\llvert \bigl\llvert f(y)
\bigr\rrvert - \bigl\llvert f(x) \bigr\rrvert \bigr\rrvert
\\[-2pt]
& \le & C \bigl\llvert f(y) -f(x) \bigr\rrvert
\\[-2pt]
&\le & C \llvert y-x \rrvert ,
\end{eqnarray*}
where in the last inequality we used Lemma \ref{Majf}.
\end{proof}

\begin{lemma}
\label{mine}
For every $A > 0$ and $t\in [0,T]$, there exists a positive
constant $c$ such that
\begin{equation*}
\int_{0}^{t} \frac{1}{2 \pi u} \int
_{{ \mathbb{R}}} \bigl\llvert E^{-}(t-s,y,z)-E
^{-}(t-s,x,z) \bigr\rrvert ^{2} \,dz\; du\geqslant \, c
\; \llvert y-x \rrvert
\end{equation*}
and
\begin{equation*}
\int_{0}^{t} \frac{1}{2 \pi u} \int
_{{ \mathbb{R}}} \bigl\llvert E^{+}(t-s,y,z)-E
^{+}(t-s,x,z) \bigr\rrvert ^{2} \,dz\; du\geqslant \, c
\; \llvert y-x \rrvert
\end{equation*}
for every $x,y \in [0,A]$.
\end{lemma}

\begin{proof}
We present the proof just for the first inequality. The second is
obtained in the same way. We have
\begin{eqnarray*}
&& \int_{0}^{t}
\frac{1}{2 \pi u} \int_{{ \mathbb{R}}} \bigl\llvert
E^{-}(t-s,y,z)-E ^{-}(t-s,x,z) \bigr\rrvert
^{2} \,dz\; du
\\
&=& \int_{0}^{t}\int_{\mathbb{R}}
\biggl[\frac{1}{\sqrt{2\pi u}} \exp \biggl(-\frac{(f(z)-f(y))^{2}}{2u} \biggr)
\\
&& {} -\frac{1}{\sqrt{2 \pi u}} \exp \biggl(- \frac{(f(z)-f(x))^{2}}{2u} \biggr)
\biggr]^{2} \,dz\;du
\\
&=& \int_{0}^{t}\int_{0}^{\infty }
\biggl[\frac{1}{\sqrt{2\pi u}} \exp \biggl(-\frac{(z/\sqrt{a_{2}}-f(y))^{2}}{2u} \biggr)
\\
&& {} -\frac{1}{\sqrt{2 \pi u}} \exp \biggl(-\frac{(z/\sqrt{a
_{2}}-f(x))^{2}}{2u} \biggr)
\biggr]^{2}\,dz\;du
\\
&& + \int_{0}^{t}\int_{-\infty }^{0}
\biggl[\frac{1}{\sqrt{2\pi u}} \exp \biggl(-\frac{(z/\sqrt{a_{1}}-f(y))^{2}}{2u} \biggr)
\\
&& {} -\frac{1}{\sqrt{2 \pi u}} \exp \biggl(-\frac{(z/\sqrt{a
_{1}}-f(x))^{2}}{2u} \biggr)
\biggr]^{2}\,dz\;du.
\end{eqnarray*}
Applying the changes of variables $Z= z/\sqrt{a_{2}}-f(x)$ in the first
integral and $Z= z/\sqrt{a_{1}}-f(x)$ in the second one we obtain
\begin{eqnarray*}
&& \int_{0}^{t}
\frac{1}{2 \pi u} \int_{{ \mathbb{R}}} \bigl\llvert
E^{-}(t-s,y,z)-E^{-}(t-s,x,z) \bigr\rrvert
^{2} \,dz\; du
\nonumber
\\
&=& \sqrt{a_{2}} \; \int_{0}^{t}
\int_{-f(x)}^{\infty } \biggl[\frac{1}{
\sqrt{2u}} \exp
\biggl(-\frac{  (Z-(f(y)-f(x))  )^{2}}{2u} \biggr)
\nonumber
\\
&& {}- \frac{1}{\sqrt{2\pi u}} \exp \biggl(-\frac{Z^{2}}{2u} \biggr)
\biggr]^{2}dZ\,du
\nonumber
\\
&+& \sqrt{a_{1}} \; \int_{0}^{t}
\int^{-f(x)}_{\infty } \biggl[\frac{1}{
\sqrt{2u}} \exp
\biggl(-\frac{  (Z-(f(y)-f(x))  )^{2}}{2u} \biggr)
\nonumber
\\
&& {}- \frac{1}{\sqrt{2\pi u}} \exp \biggl(- \frac{Z^{2}}{2u} \biggr)
\biggr]^{2},dZ\,du
\nonumber
\\
&\geqslant & \min ( \sqrt{a_{1}}, \sqrt{a_{2}} )\int
_{0}^{t} \int_{\mathbb{R}} \biggl[
\frac{1}{\sqrt{2u}} \exp \biggl(-\frac{
  (Z-(f(y)-f(x))  )^{2}}{2u} \biggr)
\nonumber
\\
&& {} - \frac{1}{\sqrt{2\pi u}} \exp \biggl(-\frac{Z^{2}}{2u} \biggr)
\biggr]^{2}\,dZ\,du
\nonumber
\\
&=& \min ( \sqrt{a_{1}}, \sqrt{a_{2}} ) \int
_{0}^{t} \int_{{ \mathbb{R}}} \bigl[
p_{u}(Z -\tilde{K}) - p_{u}(Z) \; \bigr]
^{2}\,dZ\;du, %
\end{eqnarray*}
where $\tilde{K}=f(y)-f(x)$ and $p_{u}$ is the heat kernel defined by
(\ref{GF}). Without loss of generality we can suppose that $x < y$. So,
\begin{equation*}
0< \tilde{K} = \frac{y-x}{\sqrt{a_{2}}} < \frac{A}{\sqrt{a_{2}}}.
\end{equation*}
Applying the same technique as that used in (\ref{eq:planch}), we get
\begin{eqnarray*}
&& \int_{0}^{t}\int
_{{ \mathbb{R}}} \bigl[ p_{u}(Z -h) -
p_{u}(Z) \; \bigr] ^{2}\,dZ\;du
\\
&=& \frac{1}{2\pi } \int_{0}^{t}{ds}\int
_{-\infty }^{\infty }{ \bigl\llvert e^{- s \xi ^{2}/2-i\xi h}-
e^{-s
\xi ^{2}/2} \bigr\rrvert ^{2}\;d\xi }
\\
&=& \frac{1}{\pi } \int_{0}^{t}{ds}\int
_{-\infty }^{\infty }{e^{-s \xi
^{2}} \bigl(1-\cos (h
\xi ) \bigr)\; d\xi }
\\
&=& \frac{2}{ \pi }\int_{0}^{\infty }
\bigl(1-e^{-t \xi ^{2}} \bigr) \frac{1-\cos ( \llvert h \rrvert
\xi )}{\xi ^{2}}\; \xch{d\xi } {d\xi .}
\end{eqnarray*}
for every $h \in {\mathbb{R}}$. Thus,
%
\begin{eqnarray}
\label{eq:40'} %
&& \int_{0}^{t}{ds}
\int_{\mathbb{R}}{ \bigl[p_{s}(Z-
\tilde{K})-p_{s}(Z) \bigr]^{2} dy}
\nonumber
\\
&=& \frac{2}{\pi }\int_{0}^{\infty }
\bigl(1-e^{-t z^{2}} \bigr) \frac{1-\cos (
\tilde{K} z)}{z^{2}}\; dz
\nonumber
\\
&=& \frac{2}{\pi }\int_{0}^{\frac{1}{\tilde{K}}}
\bigl(1-e^{-t z^{2}} \bigr) \frac{1-
\cos (\tilde{K} z)}{z^{2}}\; dz +
\frac{2}{\pi }\int_{\frac{1}{\tilde{K}}}^{\infty }
\bigl(1-e^{-t z^{2}} \bigr) \frac{1-
\cos (\tilde{K} z)}{z^{2}}\; dz
\nonumber
\\
&\ge & \frac{2}{\pi }\int_{\frac{1}{\tilde{K}}}^{\infty }
\bigl(1-e^{-t z^{2}} \bigr) \frac{1-
\cos (\tilde{K} z)}{z^{2}}\; dz,
\end{eqnarray}
where in the last inequality we used the fact that
\begin{equation*}
\displaystyle \frac{2}{\pi }\int_{0}^{\frac{1}{\tilde{K}}}
\bigl(1-e^{-t z^{2}} \bigr) \frac{1-
\cos (\tilde{K} z)}{z^{2}}\; dz \ge 0.
\end{equation*}
Since $1-e^{-t z^{2}} \ge 1-e^{-t \tilde{K}^{-2}}$ for every $z \ge \frac{1}{\tilde{K}}$, from (\ref{eq:40'}) we get
\begin{eqnarray*}
\int_{0}^{t}{ds}\int
_{\mathbb{R}}{ \bigl[p_{s}(Z-\tilde{K})-p_{s}(Z)
\bigr]^{2} dy} &\ge & \frac{2}{\pi } \bigl(1-e^{-t\tilde{K}^{-2}}
\bigr)\int_{\frac{1}{\tilde{K}}}^{
\infty } \frac{1-\cos (\tilde{K} z)}{z^{2}}\; dz
\\
&\ge & \tilde{K} \frac{2}{\pi } \bigl(1-e^{-t\tilde{K}^{-2}} \bigr)\int
_{1}^{\infty } \frac{1-
\cos ( \xi )}{\xi ^{2}}\; d\xi ,
\end{eqnarray*}
where the last inequality is obtained after applying the change of
variables $\xi = \tilde{K}z$.

Now, since
\begin{equation*}
0 < \tilde{K} = \frac{y-x}{\sqrt{a_{2}} } < \frac{A}{\sqrt{a_{2}}},
\end{equation*}
we have
\begin{equation*}
1-e^{-t\tilde{K}^{-2}} \ge 1-e^{-ta_{2} A^{-2}}
\end{equation*}
and consequently
\begin{equation*}
\displaystyle \int_{0}^{t}\int
_{{ \mathbb{R}}} \bigl[ p_{u}(Z -\tilde{K}) -
p_{u}(Z) \; \bigr]^{2}\,dZ\;du \geqslant c \llvert y-x
\rrvert
\end{equation*}
with
\begin{equation*}
c = \frac{2}{\sqrt{a_{2}}\pi } \bigl(1-e^{-ta_{2}A^{-2}} \bigr) \int
_{1}^{
\infty } \frac{1-\cos (z)}{z^{2}}\; dz .\qedhere
\end{equation*}
\end{proof}

\subsection{Proof of Theorem \ref{Teo.part}}%
\label{sec4.2}
Since Equation (\ref{e:2p}) is a particular case of (\ref{e:2g}), by
Theorems \ref{T:principal} and \ref{ASCLM}, to get Theorem
\ref{Teo.part}, it suffices to show that the fundamental solution\index{fundamental solution}
associated to the operator ${\mathcal{L}}_{p}$ satisfies Conditions
$H_{i}([0,1])$, for $i = 1, 2, 3$. Consider $x, y \in [0,1]; \;x < y$ and $t \in [0,T]$.

\subsubsection{Proof of $H_{1}([0,1])$}%
\label{sec4.2.1}
We have
\begin{eqnarray*}
&& \bigl\llVert \Delta _{ y-x }G(t-s,x,.) \bigr\rrVert
^{2}_{L^{2}([0,t]\times
{ \mathbb{R}})}
\\
&=& \, \int_{0}^{t}\frac{1}{2\pi (t-s)}
\biggl\{ \int_{\mathbb{R}}\frac{1}{A(z)} \bigl[
\bigl(A^{-}(t-s,y,z)-A^{-}(t-s,x,z) \bigr) \,
\mathbf{1}_{\{z \leqslant 0\}}
\\
&& {} + \bigl(A^{+}(t-s,y,z)-A^{+}(t-s,x,z) \bigr)\,
\mathbf{1} _{\{z > 0\}} \bigr]^{2} \,dz \biggr\} \,ds
\\
&=& \int_{0}^{t}\frac{1}{2\pi (t-s)} \biggl\{
\int_{\mathbb{R}}\frac{1}{A(z)} \bigl[ \bigl(E^{-}(t-s,y,z)-E^{-}(t-s,x,z)
\bigr)
\\
&& {} + \,\beta \,sign(z)\, \bigl(E^{+}(t-s,y,z)-E^{+}(t-s,x,z)
\bigr) \bigr] ^{2}\;dz \biggr\} \,ds
\\
& \geqslant & \, \min { \biggl( \frac{1}{a_{1}}, \frac{1}{a_{2}}
\biggr)}\, \int_{0}^{t}\frac{1}{2\pi (t-s)}
\biggl\{ \int_{\mathbb{R}} \bigl\llvert \bigl\llvert
E^{-}(t-s,y,z)-E^{-}(t-s,x,z)\bigr\rrvert
\\
&& {}- \llvert \beta \rrvert \,\bigl\llvert E^{+}(t-s,y,z)-E^{+}(t-s,x,z)
\bigr\rrvert \bigr\rrvert ^{2}\;dz \biggr\} \,ds.
\end{eqnarray*}

According to \cite[page 54]{DW}, we know that
\begin{equation*}
\llVert f-g \rrVert ^{2}_{L^{2}({ \mathbb{R}})}\;\geqslant \,
\frac{1}{4}\, \bigl( \llVert f \rrVert _{L^{2}({ \mathbb{R}})} + \llVert g
\rrVert _{L^{2}({ \mathbb{R}})} \bigr)^{2}\; \biggl\llVert \displaystyle
\frac{f}{ \llVert   f \rrVert   _{L^{2}({ \mathbb{R}})}
}-\frac{g }{ \llVert   g \rrVert   _{L^{2}({ \mathbb{R}})}} \biggr\rrVert ^{2}_{L^{2}( { \mathbb{R}})}
\end{equation*}
for every $f, g\in L^{2}({ \mathbb{R}}); f \neq 0$ and $g \neq 0$ a.e. Thus,
\begin{eqnarray*}
&& \bigl\llVert \Delta _{ y-x }G(t-s,x,.) \bigr\rrVert
^{2}_{L^{2}([0,t]\times{ \mathbb{R}})}
\\
&\geqslant& \frac{1}{4}\, \min { \biggl(\frac{1}{a_{1}},
\frac{1}{a_{2}} \biggr)}\, \int_{0}^{t}
\frac{I(s)}{2\pi (t-s)}
\\
&&{}\times \bigl( \bigl\llVert E^{-}(t-s,y,.)-E^{-}(t-s,x,.)
\bigr\rrVert
\\
&&\quad {} +\, \llvert \beta \rrvert \, \bigl\llVert E^{+}(t-s,y,.)-E^{+}(t-s,x,.)
\bigr\rrVert \bigr)^{2} ds,
\end{eqnarray*}
where
\begin{eqnarray*}
I(s) &=& \biggl\llVert \frac{  \llvert E^{-}(t-s,y,.)-E^{-}(t-s,x,.) \rrvert }{ \llVert   E
^{-}(t-s,y,.)-E^{-}(t-s,x,.) \rrVert   }
\\
&& {}- \llvert \beta \rrvert \,\frac{  \llvert E^{+}(t-s,y,.)-E^{+}(t-s,x,.) \rrvert }{  \llVert
E^{+}(t-s,y,.)-E^{+}(t-s,x,.) \rrVert   } \biggr\rrVert
^{2},
\end{eqnarray*}
and $\| . \|$ denotes $\| . \|_{L^{2}({\mathbb{R}})}$. On the one hand we have
\begin{eqnarray*}
&&\!\!\!\! I(s) = 1+\beta ^{2}
\\
&&\!\!\!\! {}- \frac{2\, \llvert \beta  \rrvert }{  \llVert   E^{-}(t-s,y,z)-E^{-}(t-s,x,z) \rrVert
\, \llVert   E^{+}(t-s,y,z)-E^{+}(t-s,x,z) \rrVert   }
\\
&&\!\!\!\! {}\times \int_{ \mathbb{R}} \bigl\llvert
E^{-}(t-s,y,z)-E^{-}(t-s,x,z) \bigr\rrvert \; \bigl
\llvert E^{+}(t-s,y,z)-E ^{+}(t-s,x,z) \bigr\rrvert
\,dz. %
\end{eqnarray*}

On the other hand, applying H\"{o}lder's Inequality, we get
\begin{eqnarray*}
&&\!\!\!\!\int_{\mathbb{R}} \bigl\llvert E^{-}(t-s,y,z)-E^{-}(t-s,x,z)
\bigr\rrvert \; \, \bigl\llvert E^{+}(t-s,y,z)-E^{+}(t-s,x,z)
\bigr\rrvert \;dz
\\
&&\quad \le \bigl\llVert E^{-}(t-s,y,.)-E^{-}(t-s,x,.)
\bigr\rrVert \, \bigl\llVert E^{+}(t-s,y,.)-E ^{+}(t-s,x,.)
\bigr\rrVert.
\end{eqnarray*}
Hence,
\begin{equation*}
I(s) \geqslant 1+\beta ^{2}-2 \llvert \beta \rrvert = \bigl(1-
\llvert \beta \rrvert \bigr)^{2}
\end{equation*}
and therefore,
\begin{eqnarray*}
&& \bigl\llVert \Delta _{ y-x }G(t-s,x,.) \bigr\rrVert
^{2}_{L^{2}([0,t]\times
{ \mathbb{R}})}
\\
&\geqslant & \,\frac{ (1- \llvert \beta  \rrvert )^{2} }{4}\, \min { \biggl( \frac{1}{a_{1}},
\frac{1}{a_{2}} \biggr)}\, \int_{0}^{t}
\frac{1}{2\pi (t-s)}
\\
&&{}\times \bigl( \bigl\llVert E^{-}(t-s,y,.)-E
^{-}(t-s,x,.) \bigr\rrVert
\\
&&{} +\, \llvert \beta \rrvert \, \bigl\llVert E^{+}(t-s,y,.)-E^{+}(t-s,x,.)
\bigr\rrVert \bigr)^{2} ds
\\
& \geqslant & \frac{ (1- \llvert \beta  \rrvert )^{2} }{4}\, \min { \biggl( \frac{1}{a_{1}},
\frac{1}{a_{2}} \biggr)}\, \int_{0}^{t}
\frac{1}{2\pi u}
\\
&& {}\times \; \bigl( \bigl\llVert E^{-}(u,y,.)-E^{-}(u,x,.)
\bigr\rrVert ^{2}+\, \beta ^{2}\, \bigl\llVert
E^{+}(u,y,.)-E^{+}(u,x,.) \bigr\rrVert ^{2}
\bigr) du, %
\end{eqnarray*}
where in the last inequality we used the fact that $x^{2} + y^{2} \le (x+y)^{2}$ for every non-negative real numbers
$x$ and $y$.

This and Lemma \ref{mine} show that Hypothesis $H_{1}([0,1])$ is
satisfied.

\subsubsection{Proof of $H_{2}([0,1])$}%
\label{sec4.2.2}
Using the expressions of $A^{-}$ and $A^{+}$ given in (\ref{eq:A}) we
get
%
\begin{eqnarray}
\label{ine} %
&& \bigl\llVert \Delta _{y-x }G(t-.,x,.)
\bigr\rrVert ^{2}_{L^{2}([0,t]\times
{ \mathbb{R}})}
\nonumber
\\
&=& \int_{0}^{t} \int_{\mathbb{R}}
\bigl\llvert G(t-s,y,z)-G(t-s,x,z) \bigr\rrvert ^{2} \,ds \,dz
\nonumber
\\
&=& \int_{0}^{t} \Biggl[\frac{1}{2a_{1}\pi (t-s)}
\int_{-\infty }^{0} \bigl\llvert A^{-}(t-s,y,z)-A
^{-}(t-s,x,z) \bigr\rrvert ^{2} dz \Biggr] ds
\nonumber
\\
&& {}+ \int_{0}^{t} \Biggl[
\frac{1}{2a_{2}\pi (t-s)}\int_{0}^{\infty } \bigl\llvert
A^{+}(t-s,y,z)-A ^{+}(t-s,x,z) \bigr\rrvert
^{2} dz \Biggr] \,ds
\nonumber
\\
&\le & \int_{0}^{t} \frac{1}{2\pi (t-s)}\int
_{-\infty }^{0} \Delta _{max} (s,z) dz ds
\nonumber
\\
&&{}+ \int_{0}^{t} \frac{1}{2\pi (t-s)} \int
_{0}^{\infty } \Delta _{max}(s,z) dz \,ds,
\end{eqnarray}
where
\begin{eqnarray*}
\Delta _{\max }(s,z) &=& \max \biggl( \frac{1}{a_{1}}
\; \bigl\llvert A^{-}(t-s,y,z)-A^{-}(t-s,x,z) \bigr
\rrvert ^{2},
\\
&& \frac{1}{a_{2}}\; \bigl\llvert A^{+}(t-s,y,z)-A^{+}(t-s,x,z)
\bigr\rrvert ^{2} \biggr)
\\
&=& \max \biggl( \frac{1}{a_{1}} \bigl( \bigl( E^{-}(t-s,y,z)-
E^{-}(t-s,x,z) \bigr)
\\
&& {} - \beta \bigl( E^{+}(t-s,y,z)- E^{+}(t-s,x,z)
\bigr) \bigr) ^{2} ,
\\
&& \frac{1}{a_{2}} \bigl( \bigl( E^{-}(t-s,y,z)-
E^{-}(t-s,x,z) \bigr)
\\
&& {} + \beta \bigl( E^{+}(t-s,y,z)- E^{+}(t-s,x,z)
\bigr) \bigr) ^{2} \biggr)
\end{eqnarray*}
for every $t \in (0,T]$ and $(x,y) \in [0,1]^{2}; \; y > x$.

Since
\begin{equation*}
\max \bigl(\gamma _{1}(a-b)^{2},\gamma
_{2}(a+b)^{2} \bigr)\leqslant 2 \max (\gamma
_{1}, \gamma _{2}) \bigl(a^{2}+b^{2}
\bigr)\vadjust{\goodbreak}
\end{equation*}
for any $(a,b)\in { \mathbb{R}}^{2}$ and any $\gamma _{1},\gamma _{2}>0$,
we have
\begin{eqnarray*}
\Delta _{max } (s,z) &\le& 2 \max \biggl(\frac{1}{{a_{1}}},
\frac{1}{
{a_{2}}} \biggr) \bigl( \bigl\llvert E^{-}(t-s,y,z)-E^{-}(t-s,x,z)
\bigr\rrvert ^{2}
\\
&&{}+\beta ^{2} \bigl\llvert E^{+}(t-s,y,z)-E^{+}(t-s,x,z)
\bigr\rrvert ^{2} \bigr).
\end{eqnarray*}

Thus,
%
\begin{eqnarray}
\label{ine_1} %
&& \bigl\llVert \Delta _{y-x }G(t-.,x,.)
\bigr\rrVert ^{2}_{L^{2}([0,t]\times{ \mathbb{R}})}
\nonumber
\\[-2pt]
& \le & \int_{0}^{t} \frac{1}{2\pi (t-s)}\int
_{{\mathbb{R}}} \Delta _{max} (s,z)dz \xch{ds} {ds.}
\nonumber
\\[-2pt]
&\le & 2 \max \biggl( \frac{1}{{a_{1}}},\frac{1}{{a_{2}}} \biggr)
\Biggl[ \int_{0}^{t} \biggl[
\frac{1}{2\pi (t-s)} \int_{{ \mathbb{R}}} \bigl\llvert
E^{-}(t-s,y,z)
\nonumber
\\[-2pt]
&&{}-E ^{-}(t-s,x,z) \bigr\rrvert ^{2}dz \biggr]ds
\nonumber
\\[-2pt]
&& {}+ \beta ^{2}\, \int_{0}^{t}
\biggl[\frac{1}{2\pi (t-s)} \int_{{
\mathbb{R}}} \bigl\llvert
E^{+}(t-s,y,z)-E^{+}(t-s,x,z) \bigr\rrvert
^{2}dz \biggr]ds \Biggr] .\hspace*{21pt} %
\end{eqnarray}

This and Lemma \ref{lem:UBp1} show that Condition $H_{2}([0,1])$ is
also satisfied.

\subsubsection{Proof of $H_{3}([0,1])$}%
\label{sec4.2.3}
Consider $x, x' \in [0,1]$ and $h > 0$. We have
%
\begin{eqnarray}
\label{E1} %
&& \int_{0}^{t}
\int_{\mathbb{R}}\Delta _{h}G(t-s,x,z)\,\Delta
_{h}G \bigl(t-s,x',z \bigr)\,dz\,ds
\nonumber
\\[-2pt]
&=& \int_{0}^{t}\int_{\mathbb{R}}
\bigl(G(t-s,x+h,z)-G(t-s,x,z) \bigr)
\nonumber
\\[-2pt]
&& \bigl(G \bigl(t-s,x'+h,z \bigr)-G \bigl(t-s,x',z
\bigr) \bigr) dz \,ds
\nonumber
\\[-2pt]
&=& \int_{0}^{t} \Biggl[ \frac{m^{2}(t-s)}{{a_{1}}}
\int_{-\infty }^{0} \bigl(A^{-}(t-s,x+h,z)-A^{-}(t-s,x,z)
\bigr)
\nonumber
\\[-2pt]
&& \bigl(A^{-} \bigl(t-s,x'+h,z
\bigr)-A^{-} \bigl(t-s,x',z \bigr) \bigr) dz \Biggr]ds
\nonumber
\\[-2pt]
&+& \int_{0}^{t} \Biggl[ \frac{m^{2}(t-s)}{{a_{2}}}
\int_{0}^{\infty } \bigl(A^{+}(t-s,x+h,z)-A^{+}(t-s,x,z)
\bigr)
\nonumber
\\[-2pt]
&& \bigl(A^{+} \bigl(t-s,x'+h,z
\bigr)-A^{+} \bigl(t-s,x',z \bigr) \bigr)dz \Biggr] ds
\nonumber
\\[-2pt]
&=& L + K. %
\end{eqnarray}
Using the expression of $A^{-}$ (\ref{eq:A}), denoting $\tilde{x}= \frac{x}{\sqrt{a_{2}}}$,
$\tilde{x}'= \frac{x'}{\sqrt{a_{2}}}$ and $\tilde{h}= \frac{h}{\sqrt{a_{2}}}$, then making the change of
variables $z'= \frac{ z}{\sqrt{a_{1}}}$, we get
\begingroup
\abovedisplayskip=12.5pt
\belowdisplayskip=12.5pt
\begin{eqnarray*}
L & =& \int_{0}^{t} \Biggl[
\frac{(1-\beta )^{2}}{{2\pi a_{1}u}}\int_{-\infty
}^{0}
\bigl(E^{-}(u,x+h,z)-E^{-}(u,x,z) \bigr)
\\[-1pt]
&& \bigl(E^{-} \bigl(u,x'+h,z \bigr)-E^{-}
\bigl(u,x',z \bigr) \bigr) dz \Biggr]du
\\[-1pt]
&=& \frac{1}{\sqrt{a_{2}}}\int_{0}^{t} \sqrt{
\frac{a_{2}}{a_{1}}}\frac{(1-
\beta )^{2}}{\,2\pi u} \Biggl[\int^{0}_{-\infty }
\biggl(\exp \biggl(-\frac{(z'-
\tilde{x}-\tilde{h})^{2}}{2u} \biggr)
\\[-1pt]
&& {}-\exp \biggl(-\frac{(z'-\tilde{x})^{2}}{2u} \biggr) \biggr) \times \biggl(\exp
\biggl(-\frac{(z'-\tilde{x}'- \tilde{h} )^{2}}{2u} \biggr)
\\[-1pt]
&& {}-\exp \biggl(-\frac{(z'-\tilde{x}')^{2}}{2u} \biggr) \biggr) dz'\,
\Biggr]du.
\end{eqnarray*}

Now, using the expression of $A^{+}$ (\ref{eq:A}) and making the change
of variable $z'=\frac{z}{\sqrt{a_{2}}}$, the integral $K$ can be
written in the form
\begin{equation*}
K=\,\int_{0}^{t}\frac{1}{{2\sqrt{a_{2}}\,\pi u}} \bigl[
K_{1}+ \beta \,K_{2}+\beta \,K_{3}+\beta
^{2}\,K_{4} \bigr]\;du,
\end{equation*}
where
\begin{eqnarray*}
K_{1}&=&\int_{0}^{+\infty } \biggl(\exp
\biggl(-\frac{(z'-\tilde{x}-
\tilde{h})^{2}}{2u} \biggr)-\exp \biggl(- \frac{(z'-\tilde{x})^{2}}{2u} \biggr)
\biggr)
\\[-1pt]
&&{}\times \biggl(\exp \biggl(-\frac{(z'-\tilde{x}'-\tilde{h} )^{2}}{2u} \biggr)- \exp \biggl(-
\frac{(z'-\tilde{x}')^{2}}{2u} \biggr) \biggr) dz',
\\[-1pt]
K_{2}&=&\int_{0}^{+\infty } \biggl(\exp
\biggl(-\frac{(z'-\tilde{x}-
\tilde{h})^{2}}{2u} \biggr)-\exp \biggl(- \frac{(z'-\tilde{x})^{2}}{2u} \biggr)
\biggr)
\\[-1pt]
&&{}\times \biggl(\exp \biggl(-\frac{(z'+\tilde{x}'+\tilde{h})^{2}}{2u} \biggr)- \exp \biggl(-
\frac{(z'+\tilde{x}')^{2}}{2u} \biggr) \biggr) dz',
\\[-1pt]
K_{3}&=&\int_{0}^{+\infty } \biggl(\exp
\biggl(-\frac{(z'+\tilde{x}+
\tilde{h})^{2}}{2u} \biggr)-\exp \biggl(- \frac{(z'+\tilde{x})^{2}}{2u} \biggr)
\biggr)
\\[-1pt]
&&{}\times \biggl(\exp \biggl(-\frac{(z'-\tilde{x}'-\tilde{h})^{2}}{2u} \biggr)- \exp \biggl(-
\frac{(z'-\tilde{x}')^{2}}{2u} \biggr) \biggr) dz',
\\[-1pt]
K_{4}&=&\int_{0}^{+\infty } \biggl(\exp
\biggl(-\frac{(z'+\tilde{x}+
\tilde{h})^{2}}{2u} \biggr)-\exp \biggl(- \frac{(z'+\tilde{x})^{2}}{2u} \biggr)
\biggr)
\\[-1pt]
&&{}\times \biggl(\exp \biggl(-\frac{(z'+\tilde{x}'+\tilde{h})^{2}}{2u} \biggr)- \exp \biggl(-
\frac{(z'+\tilde{x}')^{2}}{2u} \biggr) \biggr) dz'.
\end{eqnarray*}
By using the change of variable $z=-z'$, we get
\begin{eqnarray*}
K_{2}+K_{3}&=&\int_{ \mathbb{R}} \biggl(
\exp \biggl(-\frac{(z'+\tilde{x}+
\tilde{h})^{2}}{2u} \biggr)-\exp \biggl(- \frac{(z'+\tilde{x})^{2}}{2u}
\biggr) \biggr)
\\
&&{}\times \biggl(\exp \biggl(-\frac{(z'-\tilde{x}'-\tilde{h})^{2}}{2u} \biggr)- \exp \biggl(-
\frac{(z'-\tilde{x}')^{2}}{2u} \biggr) \biggr) dz'
\end{eqnarray*}
\endgroup
and
\begin{eqnarray*}
&& K_{1}+\beta ^{2}\,K_{4}
\\
&=& \int_{ \mathbb{R}} \biggl(\exp \biggl(- \frac{(z'-\tilde{x}-\tilde{h})^{2}}{2u}
\biggr)-\exp \biggl(-\frac{(z'-
\tilde{x})^{2}}{2u} \biggr) \biggr)
\\
&& {}\times \biggl(\exp \biggl(-\frac{(z'-\tilde{x}'-\tilde{h} )^{2}}{2u} \biggr)- \exp \biggl(-
\frac{(z'-\tilde{x}')^{2}}{2u} \biggr) \biggr) dz'
\\
&& {}+ \bigl(\beta ^{2}-1 \bigr) \int^{0}_{-\infty }
\biggl(\exp \biggl(-\frac{(z'-
\tilde{x}-\tilde{h})^{2}}{2u} \biggr)-\exp \biggl(-
\frac{(z'-
\tilde{x})^{2}}{2u} \biggr) \biggr)\,
\\
&& {}\times \biggl(\exp \biggl(-\frac{(z'-\tilde{x}'-\tilde{h} )^{2}}{2u} \biggr)- \exp \biggl(-
\frac{(z'-\tilde{x}')^{2}}{2u} \biggr) \biggr) dz'.
\end{eqnarray*}
Therefore,
%
\begin{equation}
\label{E2} L+ K = L_{1} + \beta L_{2} + \biggl(
\sqrt{\frac{a_{2}}{a_{1}}}(1- \beta )^{2}+\beta ^{2}-1
\biggr) L_{3},
\end{equation}
where
\begin{eqnarray*}
L_{1} &=& \,\int_{0}^{t}
\frac{du}{{2\sqrt{a_{2}}\,\pi u}} \, \int_{ \mathbb{R}} \biggl(\exp \biggl(-
\frac{(z'-\tilde{x}-\tilde{h})^{2}}{2u} \biggr)-\exp \biggl(-\frac{(z'-
\tilde{x})^{2}}{2u} \biggr) \biggr)\,
\\
&&{}\times \biggl(\exp \biggl(-\frac{(z'-\tilde{x}'-\tilde{h})^{2}}{2u} \biggr)- \exp \biggl(-
\frac{(z'-\tilde{x}')^{2}}{2u} \biggr) \biggr) dz',
\\
L_{2} &=& \int_{0}^{t}
\frac{du}{{2\sqrt{a_{2}}\,\pi u}} \int_{
\mathbb{R}} \biggl(\exp \biggl(-
\frac{(z'+\tilde{x}+\tilde{h})^{2}}{2u} \biggr)- \exp \biggl(-\frac{(z'+\tilde{x})^{2}}{2u} \biggr) \biggr)
\,
\\
&&{}\times \biggl(\exp \biggl(-\frac{(z'-\tilde{x}'-\tilde{h})^{2}}{2u} \biggr)- \exp \biggl(-
\frac{(z'-\tilde{x}')^{2}}{2u} \biggr) \biggr) dz'
\end{eqnarray*}
and
\begin{eqnarray*}
L_{3} &=& \int_{0}^{t}
\frac{du}{{2\sqrt{a_{2}}\,\pi u}} \int_{-\infty
}^{0} \biggl(\exp
\biggl(-\frac{(z'-\tilde{x}-\tilde{h})^{2}}{2u} \biggr)- \exp \biggl(-\frac{(z'-\tilde{x})^{2}}{2u} \biggr)
\biggr)\,
\\
&&{}\times \biggl(\exp \biggl(-\frac{(z'-\tilde{x}'-\tilde{h})^{2}}{2u} \biggr)- \exp \biggl(-
\frac{(z'-\tilde{x}')^{2}}{2u} \biggr) \biggr) dz'.
\end{eqnarray*}

We first investigate the sign of the third integral. On the one hand,
$z' \le 0$, $\tilde{x} \ge 0$ and $\tilde{h} \ge 0$; thus, by virtue of
the fact that the function $x \longmapsto \exp (-x^{2})$ is increasing
on the interval $(- \infty , 0]$, we see that $L_{3} \ge 0$. On the
other hand, using the expression of $\beta $, given in (\ref{eq:f}),
by the fact that $\rho _{2} \ge \rho _{1}$ (see Condition
(\ref{CrhoA})), we get
\begin{equation*}
\sqrt{\frac{a_{2}}{a_{1}}}(1-\beta )^{2}+\beta ^{2}-1
= \frac{-4\rho
_{1} \sqrt{a_{1}} \sqrt{a_{2}}}{(\rho _{2} \sqrt{a_{2}} + \rho
_{1} \sqrt{a_{1}})^{2}} (\rho _{2}-\rho _{1}) \le 0.
\end{equation*}

Therefore,
%
\begin{equation}
\label{E3} \biggl( \sqrt{\frac{a_{2}}{a_{1}}}(1-\beta )^{2}+
\beta ^{2}-1 \biggr) L _{3} \le 0.
\end{equation}

Now we will calculate explicitely the integral $L_{1}$.
With the notation
%
\begin{equation}
\label{T} \mathcal{T}(x,y,u):=\int_{ \mathbb{R}}\exp
\biggl(- \frac{(v-y)^{2}}{2u} \biggr)\,\exp \biggl(-\frac{(v-x)^{2}}{2u} \biggr)
\,dv, \hspace*{1cm}
\end{equation}
for every $ x, y \in { \mathbb{R}}$ and $u > 0$, $L_{1}$ can be written
in the form
\begin{eqnarray*}
L_{1} &=& \int_{0}^{t}
\frac{1}{2\sqrt{a_{2}} \pi u} \bigl\{ \mathcal{T} \bigl( \tilde{x}+ \tilde{h},
\tilde{x}'+ \tilde{h}, u \bigr)-\mathcal{T} \bigl(\tilde{x},
\tilde{x}'+ \tilde{h}, u \bigr)
\\
& & {}-\mathcal{T} \bigl(\tilde{x} + \tilde{h}, \tilde{x}', u
\bigr)+ \mathcal{T} \bigl( \tilde{x}, \tilde{x}', u \bigr) \bigr\}
\,du.
\end{eqnarray*}

By the changes of variables $V=v-x$
and $W=\frac{y-x-2v}{2\sqrt{u}}$, we get
\begin{eqnarray*}
\mathcal{T}(x,y, u)&=& \int_{ \mathbb{R}}\exp \biggl(-
\frac{v^{2}}{2u} \biggr) \exp \biggl( \frac{-((y-x)-v)^{2}}{2u} \biggr)\;dv
\\
&=& \exp \biggl(-\frac{(y-x)^{2}}{4u} \biggr)\,\int_{ \mathbb{R}}
\exp \biggl(-\frac{((y-x)-2v)^{2}}{4u} \biggr)\,dv
\\
&=& \sqrt{\pi u}\,\exp \biggl(-\frac{(y-x)^{2}}{4u} \biggr).
\end{eqnarray*}

Thus, applying an integration by parts then the change of variables
$w=\frac{y-x}{2\sqrt{u}}$, we get
\begin{eqnarray*}
&& \int_{0}^{t}
\frac{1}{2\pi u}\mathcal{T}(x,y,u)\,du:=\int_{0}^{t}
\frac{1}{2\sqrt{
\pi u}}\,\exp \biggl(-\frac{(y-x)^{2}}{4u} \biggr)\,du
\\
&=& \sqrt{\frac{t}{\pi }}\,\exp \biggl(-\frac{(y-x)^{2}}{4t} \biggr) -
\frac{(y-x)^{2}}{4\sqrt{
\pi }}\int_{0}^{t}
u^{-3/2}\,\exp \biggl(-\frac{(y-x)^{2}}{4u} \biggr) \,du
\\
&=& \sqrt{\frac{t}{\pi }}\,\exp \biggl(-\frac{(y-x)^{2}}{4t} \biggr)- \,
\frac{1}{2}(y-x)\,\mathbf{erfc} \biggl(\frac{y-x}{2\sqrt{t}} \biggr) .
\end{eqnarray*}
Hence,
\begin{eqnarray*}
L_{1} &=& \sqrt{\frac{t}{a_{2}\pi }} \biggl\{2\,\exp
\biggl( -\frac{(\tilde{x}'-
\tilde{x})^{2}}{4t} \biggr)-\exp \biggl( -\frac{(\tilde{x}'-
\tilde{x} + \tilde{h})^{2}}{4t} \biggr)
\\
&& {}-\exp \biggl( -\frac{(\tilde{x}'-\tilde{x}-\tilde{h})^{2}}{4t} \biggr) \biggr\}
\\
&& {}-\frac{1}{2\sqrt{a_{2}}} \biggl\{ 2\, \bigl( \tilde{x}'-
\tilde{x} \bigr)\, \mathbf{erfc} \biggl( \frac{\tilde{x}'- \tilde{x}}{2\sqrt{t}} \biggr) - \bigl(
\tilde{x}'- \tilde{x} + \tilde{h} \bigr) \,\mathbf{erfc} \biggl(
\frac{
\tilde{x}'- \tilde{x} + \tilde{h}}{2\sqrt{t}} \biggr) \,
\\
&& {}- \bigl(\tilde{x}'- \tilde{x} - \tilde{h} \bigr)
\mathbf{erfc} \biggl( \frac{
\tilde{x}'- \tilde{x} - \tilde{h}}{2\sqrt{t}} \biggr) \biggr\}.
\end{eqnarray*}
The function $\tilde{h} \longmapsto L_{1} (\tilde{h}) = L_{1}$ is clearly twice
differentiable and via a simple calculation we get
\begin{equation*}
L_{1}'(\tilde{h}) = -\frac{1}{2\sqrt{a_{2}}} \biggl\{
\mathbf{erfc} \biggl( \frac{ \tilde{x}'- \tilde{x} - \tilde{h}}{2\sqrt{t}} \biggr) - \mathbf{erfc} \biggl(
\frac{\tilde{x}'- \tilde{x}+ \tilde{h}}{2
\sqrt{t}} \biggr) \biggr\}
\end{equation*}
and
\begin{equation*}
L_{1}''(\tilde{h})=
\frac{-1}{2\sqrt{a_{2}}\sqrt{\pi t}} \biggl[ \exp \biggl( -\frac{( \tilde{x}'- \tilde{x}-h)^{2}}{4t} \biggr) + \exp
\biggl( -\frac{(\tilde{x}'- \tilde{x} +h)^{2}}{4t} \biggr) \biggr] .
\end{equation*}
It's easy to check that $L_{1}(0)= L_{1}'(0)=0$ and that $L_{1}''$ is
bounded. From all this and by Taylor's formula, we obtain
%
\begin{equation}
\label{E4} L_{1}( \tilde{h})\leqslant C \,
\tilde{h}^{2} \le C h^{2}
\end{equation}
for every $h > 0$, where $C $ denotes a positive universal constant.

Applying the same techniques used in the above 
argunents, and since
$\beta \ge 0$
(see the expression of $\beta $ given in (\ref{eq:f}) and
Assumption (\ref{CrhoA}\xch{))}{)))}, we get
%
\begin{equation}
\label{E5} \beta L_{2}( \tilde{h})\leqslant C \,
\tilde{h}^{2} \le C h^{2}
\end{equation}
for every $h > 0$. Combining \xch{(\ref{E1}), (\ref{E}), (\ref{E3}), (\ref{E4})}{(\ref{E1}),(\ref{E}), (\ref{E3}),(\ref{E4})}
and (\ref{E5}), the proof of $H_{3}([0,1])$ is finished, and
consequently, the proof of Theorem \ref{Teo.part} is also finshed.

\begin{remark}
Considering an integer $d \ge 1$, one can extend Equation (\ref{e:2g})
to the $d$-dimensional case by introducing the following SPDE:
%
\begin{equation}
\label{e:2g-d} \lleft \{ %
\begin{array}{@{}rcl}
\displaystyle
\frac{\partial u(t, x)}{\partial t} &=&
\displaystyle
{\mathcal{L}}_{d} u(t, x) +\dot{W}_{d} (t,x) ;
\quad
t > 0, \; x = (x_{1}, \ldots, x_{d}) \in { \mathbb{R}}^{d} ,
\\
\noalign{\vskip2mm}
\displaystyle
u(0,.) &:= &0,\;
\end{array} %
\rright .
\end{equation}
with
%
\begin{equation}
\label{e:d-dimensional} {\mathcal{L}}_{d} = \sum
_{i,j =1}^{d} \frac{1}{r_{ij}(x)}
\frac{
\partial }{\partial x_{i}} \biggl( R_{ij}(x) \frac{\partial }{\partial
x_{j}} \biggr)
,
\end{equation}
where $x \longmapsto R_{ij}(x) $ and $x \longmapsto r_{ij}(x) $ are two measurable and bounded real-valued
functions satisfying
\begin{equation*}
r_{ij}(x) = r_{ji}(x), \qquad R_{ij}(x) =
R_{ji}(x),
\end{equation*}
and there exists a constant $\nu > 0$ such that
%
\begin{equation}
\label{hy} r_{ij}(x) \xi _{i} \xi
_{j} \ge \nu \llVert \xi \rrVert _{d}^{2}
\quad \mbox{and} \quad R_{ij}(x) \xi _{i} \xi
_{j} \ge \nu \llVert \xi \rrVert _{d}^{2}
\end{equation}
for every $x \in { \mathbb{R}}^{d}$, $\xi = (\xi _{1},\ldots, \xi _{d})
\in { \mathbb{R}}^{d}$ and $i, j \in \{ 1, \ldots, d\} $. In (\ref{hy}),
$\| . \|_{d}$ denotes the Euclidean norm in ${ \mathbb{R}}^{d}$, and in
(\ref{e:d-dimensional}) $\frac{\partial }{\partial x_{i}}$ denotes the partial derivative in the
distributional sense. The noise $W_{d}$ is a centered Gaussian field
$W_{d} = \{ W_{d}(t,C); t \in [0,T ], C \in B_{b}({\mathbb{R}}^{d}) \}$
with covariance
%
\begin{equation}
\label{e:3} {\mathbb{E}} \bigl(W_{d}(t,C)W_{d}(s,D)
\bigr) = (t \wedge s) \lambda _{d} (C \cap D),
\end{equation}
where $ \lambda _{d}$ denotes the Lebesgue measure on ${ \mathbb{R}}
^{d}$ and $B_{b}({\mathbb{R}}^{d})$ is the set of $\lambda _{d}$-bounded
Borel sub-sets of ${\mathbb{R}}^{d}$. In the particular case where
$d=1$, SPDE\index{stochastic partial differential equation (SPDE)} (\ref{e:2g-d}) is clearly reduced to Equation (\ref{e:2g}).

According to \cite{Aronson}, if we denote by $G_{d}$ the
fundamental solution\index{fundamental solution} associated to the operator ${\mathcal{L}}_{d} $,
then there exist two constants $D_{1} > 0$ and $D_{2} > 0$ such that
\begin{equation*}
G_{d}(t,x, y) \ge \frac{D_{1}}{t^{d/2}} \exp \biggl( -
\frac{D_{2}  \llVert  x-y
 \rrVert _{d}^{2}}{t} \biggr)
\end{equation*}
for every $t \in [0,T]$ and $(x,y)\in {\mathbb{R}}^{d}$. It follows then
that
%
\begin{equation}
\label{eqgd} \int_{0}^{t} \int
_{{{ \mathbb{R}}}^{d}} G_{d}^{2}(t-s,x, y) dy ds
\ge \int_{0}^{t} \int_{{{ \mathbb{R}}}^{d}}
\frac{D_{1}^{2}}{(t-s)^{d}} \exp \biggl( -\frac{2D_{2}  \llVert  x-y  \rrVert _{d}
^{2}}{t-s} \biggr) dy ds.
\end{equation}
Denoting by $I$ the right-hand side of Inequality (\ref{eqgd}), we have
%
\begin{eqnarray}
\label{eqgd2} %
I &=& \int_{0}^{t}
\frac{D_{1}^{2}}{(t-s)^{d}} \prod_{i=1}^{d}
\biggl( \int_{{{ \mathbb{R}}}} \exp \biggl( - \frac{2D_{2} (x_{i}-y_{i})^{2}}{t-s}
\biggr) dy_{i} \biggr) ds
\nonumber
\\
&=& \int_{0}^{t} \frac{D_{1}^{2}}{(t-s)^{d}} \biggl(
\sqrt{\frac{\pi (t-s)}{2D
_{2}}} \biggr)^{d} ds
\nonumber
\\
&=& D_{1}^{2} \biggl( \frac{\pi }{2D_{2}}
\biggr)^{d/2} \int_{0}^{t}
\frac{ds}{(t-s)^{d/2}},
\end{eqnarray}
where the second equality in (\ref{eqgd2}) is obtained by the change of
variables
$y_{i}' = (x_{i}-y_{i}) \sqrt{\frac{2D_{2}}{t-s}}$. Since the term
$\int_{0}^{t} \frac{ds}{(t-s)^{d/2}}$ is finite if, and only if
$d < 2$, from (\ref{eqgd}) and (\ref{eqgd2}) we deduce that, for every
$d \ge 2$ we have
\begin{equation*}
\int_{0}^{t} \int_{{{ \mathbb{R}}}^{d}}
G_{d}^{2}(t-s,x, y) dy ds = + \infty .
\end{equation*}

Therefore, for $d \ge 2$, the Wiener integral $\int_{0} ^{t} \int_{\mathbb{R}^{d} } G_{d}(t-s, x, y) W_{d}(ds, dy)$\index{Wiener  integral} is
not well-defined and consequently, the mild solution\index{mild solution} to Equation
(\ref{e:2g-d}) exists if, and only if $d= 1$.
\end{remark}


\begin{acknowledgement}
We thank very much the unknown referee for his valuable and profound
comments.
\end{acknowledgement}


\end{document}